\documentclass[10pt]{article}

\usepackage{etoolbox}
\usepackage{comment}
\newtoggle{colt}
\togglefalse{colt}
\newtoggle{icml}
\togglefalse{icml}
\newcommand{\colt}[1]{\iftoggle{colt}{#1}{}}
\newcommand{\arxiv}[1]{\iftoggle{colt}{}{#1}}

\newcommand{\icml}[1]{\iftoggle{icml}{#1}{}}

\colt{\usepackage{times}}

\arxiv{
\usepackage[letterpaper, left=1in, right=1in, top=1in,bottom=1in]{geometry}
  \usepackage{parskip}
  \usepackage[colorlinks=true, linkcolor=blue!70!black, citecolor=blue!70!black,urlcolor=black,breaklinks=true]{hyperref}
  \usepackage[dvipsnames]{xcolor}
}

\icml{
  \usepackage{parskip}
    \usepackage[colorlinks=true, linkcolor=blue!70!black, citecolor=blue!70!black,urlcolor=black,breaklinks=true, bookmarks=true]{hyperref}
}

\PassOptionsToPackage{hypertexnames=false}{hyperref}  %

\usepackage{amsmath}
\usepackage{microtype}
\usepackage{hhline}

\arxiv{
\usepackage{amsthm}
}

\usepackage{bbm}
\usepackage{amsfonts}
\usepackage{amssymb}
\usepackage[nameinlink,capitalize]{cleveref}

\usepackage{mathtools}

\usepackage{xargs}

\arxiv{
\usepackage{algorithm}

\usepackage{natbib}
\bibliographystyle{plainnat}
\bibpunct{(}{)}{;}{a}{,}{,}
}

\usepackage{xpatch}

\arxiv{
\theoremstyle{plain}
\newtheorem{theorem}{Theorem}[section]
\newtheorem{proposition}[theorem]{Proposition}
\newtheorem{lemma}[theorem]{Lemma}
\newtheorem{corollary}[theorem]{Corollary}
\newtheorem{remark}{Remark}

\newtheorem{definition}{Definition}
\newtheorem{assumption}{Assumption}

\theoremstyle{definition}
\newtheorem{example}{Example}

}
\colt{
\newcommand{\qed}{\hfill\ensuremath{\Box}}
\newenvironment{proof}{
\textbf{Proof.} 
}{
\qed
}

\usepackage{thmtools}

\newtheorem{theorem}{Theorem}[section]
\newtheorem{lemma}[theorem]{Lemma}
\newtheorem{definition}[theorem]{Definition}
\newtheorem{corollary}[theorem]{Corollary}
\newtheorem{proposition}[theorem]{Proposition}

\newtheorem{remark}[theorem]{Remark}
\newtheorem{example}[theorem]{Example}

\newtheorem{assumption}[theorem]{Assumption}

}

\newcommand{\pfref}[1]{Proof of \cref{#1}}

\renewcommand{\eqref}[1]{\texorpdfstring{\hyperref[#1]{(\ref*{#1})}}{(\ref*{#1})}}

\crefformat{equation}{#2Eq.\,(#1)#3}
\Crefformat{equation}{#2Eq.\,(#1)#3}
\Crefformat{figure}{#2Figure~#1#3}
\Crefformat{assumption}{#2Assumption~#1#3}
\Crefname{assumption}{Assumption}{Assumptions}
\crefname{fact}{Fact}{Facts}
\Crefformat{figure}{#2Figure #1#3}
\Crefformat{assumption}{#2Assumption #1#3}
\crefname{lemma}{Lemma}{Lemmas}
\Crefname{lemma}{Lemma}{Lemmas}

\usepackage{crossreftools}
\arxiv{
\makeatletter
\renewenvironment{proof}[1][Proof]%
{%
  \par\noindent{\bfseries\upshape {#1.}\ }%
}%
{\qed\newline}
\makeatother
}
\colt{
\makeatletter
\renewenvironment{proof}[1][Proof]%
{%
  \par\noindent{\bfseries\upshape {#1.}\ }%
}%
{\jmlrQED}
\makeatother
}

\xpatchcmd{\proof}{\itshape}{\normalfont\proofnameformat}{}{}
\newcommand{\proofnameformat}{\bfseries}

\usepackage[most]{tcolorbox}

\tcbset {
  base/.style={
    arc=0mm, 
    bottomtitle=0.5mm,
    boxrule=0mm,
    colbacktitle=!10!white, 
    coltitle=black, 
    fonttitle=\bfseries, 
    left=2.5mm,
    leftrule=1mm,
    right=3.5mm,
    title={#1},
    toptitle=0.75mm, 
  }
}

\newtcolorbox{mainbox}[1]{
  colframe=blue!10!black,
  colbacktitle=blue!50!black!30!white,
  colback=blue!2!white,
  enhanced,
  fonttitle=\bfseries,
  attach boxed title to top left={yshift=-2.5mm},
  boxed title style={size=small,colframe=blue!40!black,colback=blue!40!black},
  title={\small\textcolor{white}{\textsc{#1}}}
}

\newtcolorbox{minbox}[1]{
  colframe=blue!10!black,
  colbacktitle=blue!50!black!30!white,
  colback=blue!2!white,
  enhanced,
  fonttitle=\bfseries,
}

\DeclarePairedDelimiter{\abs}{\lvert}{\rvert} %
\DeclarePairedDelimiter{\brk}{[}{]}
\DeclarePairedDelimiter{\crl}{\{}{\}}
\DeclarePairedDelimiter{\prn}{(}{)}
\DeclarePairedDelimiter{\nrm}{\|}{\|}
\DeclarePairedDelimiter{\tri}{\langle}{\rangle}

\DeclarePairedDelimiter{\floor}{\lfloor}{\rfloor}

\def\ddefloop#1{\ifx\ddefloop#1\else\ddef{#1}\expandafter\ddefloop\fi}
\def\ddef#1{\expandafter\def\csname bb#1\endcsname{\ensuremath{\mathbb{#1}}}}
\ddefloop ABCDEFGHIJKLMNOPQRSTUVWXYZ\ddefloop
\def\ddefloop#1{\ifx\ddefloop#1\else\ddef{#1}\expandafter\ddefloop\fi}
\def\ddef#1{\expandafter\def\csname b#1\endcsname{\ensuremath{\mathbf{#1}}}}
\ddefloop ABCDEFGHIJKLMNOPQRSTUVWXYZ\ddefloop
\def\ddef#1{\expandafter\def\csname sf#1\endcsname{\ensuremath{\mathsf{#1}}}}
\ddefloop ABCDEFGHIJKLMNOPQRSTUVWXYZ\ddefloop
\def\ddef#1{\expandafter\def\csname c#1\endcsname{\ensuremath{\mathcal{#1}}}}
\ddefloop ABCDEFGHIJKLMNOPQRSTUVWXYZ\ddefloop
\def\ddef#1{\expandafter\def\csname h#1\endcsname{\ensuremath{\widehat{#1}}}}
\ddefloop ABCDEFGHIJKLMNOPQRSTUVWXYZ\ddefloop
\def\ddef#1{\expandafter\def\csname hc#1\endcsname{\ensuremath{\widehat{\mathcal{#1}}}}}
\ddefloop ABCDEFGHIJKLMNOPQRSTUVWXYZ\ddefloop
\def\ddef#1{\expandafter\def\csname t#1\endcsname{\ensuremath{\widetilde{#1}}}}
\ddefloop ABCDEFGHIJKLMNOPQRSTUVWXYZ\ddefloop
\def\ddef#1{\expandafter\def\csname tc#1\endcsname{\ensuremath{\widetilde{\mathcal{#1}}}}}
\ddefloop ABCDEFGHIJKLMNOPQRSTUVWXYZ\ddefloop
\def\ddef#1{\expandafter\def\csname #1#1\endcsname{\ensuremath{\mathbb{#1}}}}
\ddefloop ABCDEFGHIJKLMNOPQRSTUVWXYZ\ddefloop
\def\ddef#1{\expandafter\def\csname #1\endcsname{\ensuremath{\mathbb{#1}}}}
\ddefloop ABCDEFGHIJKLMNOPQRSTUVWXYZ\ddefloop
\def\ddef#1{\expandafter\def\csname D#1\endcsname{\ensuremath{\Delta(\mathcal{#1})}}}
\ddefloop ABCDEFGHIJKLMNOPQRSTUVWXYZ\ddefloop
\def\ddefloop#1{\ifx\ddefloop#1\else\ddef{#1}\expandafter\ddefloop\fi}
\def\ddef#1{\expandafter\def\csname scr#1\endcsname{\ensuremath{\mathscr{#1}}}}
\ddefloop ABCDEFGHIJKLMNOPQRSTUVWXYZ\ddefloop

\let\Pr\undefined

\DeclareMathOperator{\En}{\mathbb{E}}

\DeclareMathOperator{\Pr}{Pr}

\newcommand{\wt}[1]{\widetilde{#1}}
\newcommand{\wh}[1]{\widehat{#1}}
\newcommand{\wb}[1]{\widebar{#1}}

\newcommand{\Dkl}[2]{D_{\mathsf{KL}}\prn*{#1\,\|\,#2}}

\newcommand{\Dchis}[2]{D_{\chi^2}\prn*{#1\dmid{}#2}}

\newcommand{\Dtv}[2]{D_{\mathsf{TV}}\prn*{#1,#2}}

\newcommand{\Dren}[3][\lambda]{D_{#1}\prn*{#2\,\|\,#3}}

\newcommand{\eps}{\epsilon}

\newcommand{\ldef}{\vcentcolon=}

\newcommand{\pclip}[2][\B]{\mathsf{Clip}_{#1}(#2)}
\newcommand{\trunc}[2][\B]{\tau_{#1}(#2)}

\newcommandx{\gam}[3][1=x,2=z]{\gamma_{#2,#3}(#1)}
\newcommandx{\gamp}[3][1=x,2=z]{\dot\gamma_{#2,#3}(#1)}

\newcommandx{\gamz}[4][1=x,2=z,3=\lr,4=\xz]{\gamma_{#2,#3,#4}(#1)}
\newcommandx{\gamzp}[4][1=x,2=z,3=\lr,4=\xz]{\dot\gamma_{#2,#3,#4}(#1)}

\newcommand{\xz}{x_0}
\newcommand{\const}{\mathrm{const}}
\newcommand{\lr}{r}

\renewcommand{\B}{B}

\newcommand{\xstar}{x^\star}
\newcommand{\xhat}{\wh{x}}

\newcommand{\CLSI}{C_{\mathsf{LSI}}}
\newcommand{\CPI}{C_{\mathsf{PI}}}

\newcommand{\rgo}{\nu}

\newcommand{\Neta}{\normal{0,\eta\Id}}

\newcommand{\Poi}{\mathsf{Poisson}}

\newcommand{\Nsigma}{\normal{0,\sigma^2\Id}}

\newcommand{\alr}{a_{\lr}}
\newcommand{\blr}{b_{\lr}}
\newcommand{\alrp}{a_{\lr}'}
\newcommand{\blrp}{b_{\lr}'}

\newcommand{\nuhat}{\wh{\nu}}

\newcommand{\mmt}{\mathfrak{m}}

\newcommand{\epsprox}{\varepsilon_{\mathsf{prox}}}

\newcommand{\rep}[1]{^{(#1)}}
\newcommand{\Oprox}[2][\eta]{\mathsf{O}_{\mathsf{prox},\eta}(#2)}
\newcommand{\Ograd}[1]{\mathsf{O}_{\mathsf{grad}}(#1)}
\newcommand{\Oeval}[1]{\mathsf{O}_{\mathsf{eval}}(#1)}
\newcommand{\Ogradn}[1]{\mathsf{O}_{\mathsf{grad}}\rep{n}(#1)}
\newcommand{\Oevaln}[1]{\mathsf{O}_{\mathsf{eval}}\rep{n}(#1)}

\newcommand{\gvar}{\sigma_{\sf g}}

\newcommand{\leqsim}{\approxleq}

\newcommand{\Unif}{\mathsf{Unif}}

\newcommand{\var}{\mathrm{Var}}
\newcommand{\Var}{\var}

\newcommand{\approxleq}{\lesssim}

\newcommand{\Id}{I}

\newcommand{\indic}{\mathbb{I}}

\renewcommand{\Pr}{\bbP}

\newcommand{\poly}{\mathrm{poly}}
\newcommand{\polylog}{\mathrm{polylog}}

\newcommand{\Ber}{\mathsf{Ber}}

\newcommand{\dmid}{\;\|\;}

\newcommand{\unif}{\mathsf{Unif}}

\newcommand{\deq}{\coloneqq}

\newcommand{\muhat}{\widehat{\mu}}

\def\multiset#1#2{\ensuremath{\left(\kern-.3em\left(\genfrac{}{}{0pt}{}{#1}{#2}\right)\kern-.3em\right)}}

\newcommand{\norm}[1]{\left \lVert #1 \right \rVert}

\newcommand{\Alg}{\texttt{Alg}}

\input{widebar}

\usepackage[utf8]{inputenc} %
\usepackage[T1]{fontenc}    %
\usepackage{url}            %
\usepackage{booktabs}       %
\usepackage{amsfonts}       %
\usepackage{nicefrac}       %
\usepackage{microtype}      %
\usepackage{makecell}
\usepackage{enumitem}
\usepackage{breakcites}
\usepackage{mathrsfs}

\usepackage[normalem]{ulem}

\usepackage{algorithm}
\usepackage{verbatim}
\usepackage{algorithmic}

\icml{

}

\newcommand{\jmlrQED}{\qed}

\usepackage{multicol}
\usepackage{colortbl}
\usepackage{setspace}
\usepackage{transparent}
\usepackage{upgreek}

\usepackage{inconsolata}
\usepackage[scaled=.90]{helvet}
\usepackage{xspace}

\icml{
\setlist[enumerate]{leftmargin=*}
\setlist[itemize]{leftmargin=*}
}

\usepackage{graphicx}
\icml{\usepackage{subfigure}}

\usepackage[suppress]{color-edits}
\addauthor{fc}{blue}
\addauthor{sc}{orange}
\addauthor{cd}{pink}
\addauthor{sr}{red}

\usepackage{parskip}

\let\oldparagraph\paragraph

\renewcommand{\paragraph}[1]{\oldparagraph{#1.}}

\newcommand{\normal}[1]{\mathsf{N}\prn*{#1}}

\makeatletter
\g@addto@macro\appendix{%
  \crefalias{section}{appendixsection}%
  \crefalias{subsection}{appendixsubsection}%
  \crefalias{subsubsection}{appendixsubsubsection}%
}
\makeatother
\crefname{appendixsection}{Appendix}{Appendices}
\Crefname{appendixsection}{Appendix}{Appendices}
\crefname{appendixsubsection}{Appendix}{Appendices}
\Crefname{appendixsubsection}{Appendix}{Appendices}
\crefname{appendixsubsubsection}{Appendix}{Appendices}
\Crefname{appendixsubsubsection}{Appendix}{Appendices}
\crefname{assumption}{Assumption}{Assumptions}
\Crefname{assumption}{Assumption}{Assumptions}

\colt{
\title[High-accuracy log-concave sampling with stochastic gradients]{High-accuracy log-concave sampling with stochastic gradients}

\coltauthor{%
 \Name{Author Name1} \Email{abc@sample.com}\\
 \addr Address 1
 \AND
 \Name{Author Name2} \Email{xyz@sample.com}\\
 \addr Address 2%
}
}

\arxiv{
\title{High-accuracy log-concave sampling with stochastic queries}

\author{
  Fan Chen \\ {\small MIT} \\ {\small \texttt{fanchen@mit.edu}} 
  \and Sinho Chewi \\ {\small Yale University} \\ {\small \texttt{sinho.chewi@yale.edu}} 
  \and Constantinos Daskalakis \\ {\small MIT} \\ {\small \texttt{costis@csail.mit.edu}}
  \and Alexander Rakhlin \\ {\small MIT} \\ {\small \texttt{rakhlin@mit.edu}}
}
}

\begin{document}

\maketitle

\begin{abstract}%
    We show that high-accuracy guarantees for log-concave sampling---that is, iteration and query complexities which scale as $\mathrm{poly}\log(1/\delta)$, where $\delta$ is the desired target accuracy---are achievable using stochastic gradients with sub-exponential tails.
    Notably, this exhibits a separation with the problem of convex optimization, where stochasticity (even additive Gaussian noise) in the gradient oracle incurs $\mathrm{poly}(1/\delta)$ queries.
    We also give an information-theoretic argument that light-tailed stochastic gradients are necessary for high accuracy: for example, in the bounded variance case, we show that the minimax-optimal query complexity scales as $\Theta(1/\delta)$. Our framework also provides similar high-accuracy guarantees under stochastic zeroth-order (value) queries, \scedit{and an improved complexity result for sampling from finite-sum potentials.}
\end{abstract}

\section{Introduction}

We study the problem of sampling from a log-concave density $\mu \propto e^{-f}$ given access to a stochastic gradient oracle for $f$.
Our main result shows that if the stochastic gradients are unbiased and have light tails (e.g., sub-exponential), then it is possible to generate a $\delta$-accurate sample in total variation distance in $\polylog(1/\delta)$ queries and time.
We refer to such a guarantee as a \emph{high-accuracy} guarantee.

Recent works take inspiration from the close connections between log-concave sampling and the better-understood field of convex optimization. From that standpoint, the phenomenon we highlight here could be surprising.
Indeed, it is well-known that optimization in the presence of noisy gradients---even additive Gaussian noise---does not admit high-accuracy guarantees.
Information-theoretic lower bounds~\citep{ Aga+12StochasticOpt,raginsky2011information} establish that the optimal bounds are $1/\delta$ in the strongly convex case, and $1/\delta^2$ in the weakly convex case.

On the other hand, in the literature on Markov chain Monte Carlo (MCMC), there are remarkable examples of ``exact MCMC'' methods in which various components of the algorithm are replaced by unbiased estimates,  yet the resulting Markov chain remains geometrically ergodic toward the original stationary distribution.
For example, suppose that $\mu$ is the marginal distribution over a parameter $\theta$, but there is an additional latent variable $z$.
In this case, the exact density can be difficult to compute, but unbiased estimators can be produced via importance sampling. Incorporating these estimators into Metropolis--Hastings algorithms leads to the class of pseudo-marginal MCMC methods~\citep{AndRob09PseudoMarginal}, some of which are exact.

When $f$ is a finite sum of functions (e.g., negative log-likelihoods in a statistical context), there is a great need to develop samplers which make use of batched stochastic gradients, echoing the stochastic gradient revolution in machine learning.
This led to the widespread use of stochastic gradient Langevin dynamics (SGLD)~\citep{WelTeh11SGLD}; see~\citet{NemFea21SGMCMC} for a survey of recent developments.
These methods are based on discretizations of diffusions and are therefore not exact, i.e., they do not admit high-accuracy guarantees.
Other works propose minibatch variants of Metropolis--Hastings methods~\citep{Sei+18BatchMH, ZhaCooDeS20BatchMH, WuWanWon22BatchMH}, leading to tailored algorithms but often without quantitative convergence guarantees.
A notable recent advance is the work of~\citet{LeeSheTia21RGO}, which developed a high-accuracy sampler for the finite sum setting; subsequent work~\citep{gopi2022private,gopi2023algorithmic} developed high-accuracy samplers with stochastic \emph{value} (zeroth-order) queries.
\scedit{We revisit the finite-sum setting in \cref{ssec:finite_sum}.}

Our interest lies in generalizing the above observations to the black-box setting, in which no particular structure for $\mu$ is assumed except log-concavity, the starting point of most non-asymptotic analyses~\citep{Chewi26Book}, as well as generic properties of the stochastic gradient oracle.
In doing so, we aim to provide precise, non-asymptotic guarantees that incorporate recent state-of-the-art advances in sampling theory so that these guarantees are as sharp as possible.

\subsection{Contributions}

Our main contribution is the development of high-accuracy samplers in the presence of stochastic gradient noise, provided that the stochastic gradients have light tails (e.g., sub-exponential or sub-Gaussian).
As a preview of our results, suppose that the target distribution is $\alpha$-strongly log-concave and $\beta$-log-smooth.
Then, state-of-the-art guarantees for high-accuracy sampling~\citep{Chewi+21MALA, WuSchChe22MALA, FanYuaChe23ImprovedProx, AltChe24Warm, chen2026high} have established that it is possible to draw a $\delta$-accurate sample in total variation distance from $\mu \propto e^{-f}$ using
\begin{align*}
    \wt O\prn[\big]{\kappa d^{1/2}\, \polylog(1/\delta)} \qquad\text{queries to exact oracles for}~f,~\nabla f\,,
\end{align*}
where $\kappa\deq \beta/\alpha$ is the condition number of $\mu$.

A consequence of our results is that it is in fact possible to draw a sample in
\begin{align*}
    \wt O\prn[\big]{(\kappa d^{1/2} + \sigma^2/\alpha)\,\polylog(1/\delta)}\qquad\text{queries to a \emph{stochastic} oracle for}~\nabla f\,,
\end{align*}
provided that the unbiased stochastic estimates $g(x)$ of the gradient $\nabla f(x)$ satisfy the sub-exponential tail bound $\En\exp(\|g(x)-\nabla f(x)\|/\sigma) \le 2$.
We note that our main results are considerably more general, allowing for both log-concavity and log-smoothness to be relaxed, and covering noisy zeroth-order queries as well; see \cref{ssec:log-concave} for details.

This demonstrates a surprising \emph{robustness to noise} for sampling, in that $\sigma^2/\alpha$ appears additively in the final bound and does not significantly deteriorate the dependence on the target accuracy $\delta$.
As discussed above, this is a stark departure from the corresponding results in optimization, in which stochasticity quickly degrades the rates to $\poly(1/\delta)$ regardless of the tail behavior.

We further remark that the work of~\citet{ChaBarLon22Lower} established a lower bound of $\Omega(\sigma^2/\delta^2)$ in a certain regime, even under Gaussian additive noise.
In \cref{ssec:bartlett}, we explain why this does not contradict our results: their lower bound example requires the strong log-concavity parameter $\alpha$ to tend to zero with $\delta$; in fact, $\alpha \lesssim \delta^2$.
Hence, our results imply that the $\Omega(\delta^{-2})$-scaling is in fact a consequence of the target distribution being ill-conditioned.
However, this raises the question of whether one can prove a lower bound which captures the dependence on $\delta$, even when $\alpha$ remains bounded away from zero.

We resolve this question via a new lower bound that captures how the tail behavior of the stochastic gradients affects the complexity of sampling to high precision.
In particular, when we only assume that the stochastic gradients have bounded variance, our lower bound reads $\Omega(1/\delta)$.
This is actually attained by our upper bound algorithm in this setting, establishing that the optimal rate is $\Theta(1/\delta)$ under a bounded variance constraint.
More generally, if we only assume that finitely many moments of the stochastic gradient are bounded, our lower bound shows that $\Omega(1/\delta^c)$ queries are necessary for some exponent $c>0$.

Taken together, our results show that \emph{light-tailed stochastic gradients are both necessary and sufficient for high-accuracy sampling}.

\scedit{Finally, we apply our method to the finite-sum setting $f = m^{-1} \sum_{i=1}^m f_i$ and improve the complexity of high-accuracy sampling from $\widetilde O(m + \kappa\,(\sqrt{md} + d))$~\citep{LeeSheTia21RGO} to $\widetilde O(m + \kappa\sqrt{md})$; see \cref{ssec:finite_sum} for details.}

\section{Preliminaries}

We first define the stochastic gradient/value oracle and its tail behavior.

\begin{assumption}[Stochastic gradient oracle]\label{ass:stoc_grad}
For any $x\in\RR^d$, we can draw i.i.d.\ samples from a distribution $\Ograd{x}$ such that under $g\sim \Ograd{x}$, it holds that $\En[g]=\nabla f(x)$. 

We assume that there is a parameter $\mmt_1>0$ such that $\En_{g\sim \Ograd{x}}\nrm{g-\nabla f(x)}\leq \mmt_1$ for any $x\in\RR^d$.
\end{assumption}

\begin{assumption}[Stochastic value oracle]\label{ass:stoc_0th}
For any $x\in\RR^d$, we can draw i.i.d.\ samples from a distribution $\Oeval{x}$ such that under $v\sim \Oeval{x}$, it holds that $\En[v]=f(x)$. 
\end{assumption}

For any stochastic oracle $O$ and integer $n\geq 1$, we define $O\rep{n}$ to be the \emph{batch oracle} that, given input $x\in\RR^d$, returns $y=\frac{1}{n}\sum_{i=1}^n y^i$ by generating i.i.d.\ samples $y^1,\dotsc,y^n\sim O(x)$.

\begin{definition}[Oracle with $\eps$-tail]\label{def:eps-tail}
Suppose that $\eps=(\eps_n)_{n\geq 1}$ is a sequence of functions. We say an oracle $O$ is of $\eps$-tail if for any $x\in\RR^d$, $M > 0$, $n\geq 1$, it holds that under $g\sim O\rep{n}(x)$, 
\begin{align*}
    \frac{1}{M}\En\brk[\big]{\nrm{g-\En[g]}\,\indic\crl{\nrm{g-\En[g]}>M}}\leq \eps_n(M;x)\,.
\end{align*}
We also denote $\eps_n(M)\deq \sup_{x\in\RR^d} \eps_n(M;x)$.
\end{definition}
Some cases of interest are as follows.
\begin{example}[Sub-polynomial tail]\label{example:sub-poly}
Suppose that for some parameter $\zeta>0$ and $\gvar>0$, for any $x$, under $g\sim O(x)$, $\En\exp\prn[\big]{\frac{\nrm{g-\En[g]}^\zeta}{\gvar^\zeta}}\leq 2$.\footnote{The case $\zeta=2$ corresponds to a sub-Gaussian tail, and $\zeta=1$ to a sub-exponential tail.}

Then, we can choose $\eps_1(M)\leq C_\zeta\exp(-c_\zeta (M/\gvar)^\zeta)$ for $M\geq \gvar$. More generally, 
we can choose $\epsilon_n(M) \le C_\zeta \exp(-c_\zeta (\sqrt{n}M/\gvar)^{\bar{\zeta}})$ where $\bar{\zeta}\deq \min\{\zeta,2\}$.
\end{example}

\begin{example}[Polynomial tail]
\label{example:poly}
Suppose that for some $k\geq1$ and any $x$,  $\En\nrm{g-\En[g]}^{2k}\leq \sigma_{2k}^{2k}$. Then we can choose $\eps_n(M)\leq \frac{(2k)!\sigma_{2k}^{2k}}{n^kM^{2k}}$.
\end{example}

We carry out our analysis under the following H\"older continuity assumption for $\nabla f$.
It interpolates between the Lipschitz case ($s=0$) and the smooth case ($s=1$).

\begin{assumption}[H\"older continuous gradient]\label{ass:holder}
    There exists $s \in [0,1]$ and $\beta_s \ge 0$ such that $\norm{\nabla f(x) -\nabla f(y)} \le \beta_s\norm{x-y}^s$ for all $x,y\in\R^d$.
\end{assumption}

For technical convenience, we also state our results using an approximate proximal oracle.

\begin{assumption}[Approximate proximal oracle]\label{asmp:prox_apx}
Given input $\xz$, the oracle $\Oprox{\xz}$ returns $\xhat$ such that
$\nrm{\xhat+\eta\nabla f(\xhat)-\xz}\leq \eta\epsprox$.

\end{assumption}

Alternatively, if we assume that the guarantee in the assumption holds with high probability, then our results remain unchanged up to another error term in total variation distance.
The following lemma shows that the approximate proximal oracle can be implemented using the stochastic gradient oracle. 

\begin{lemma}\label{lem:implement-prox}
    Suppose that \cref{ass:holder} holds with $s\in[0,1]$ and denote $m_s=\beta_s^{1/(1+s)}$.
    Suppose that $\eta\leq \frac{1}{2m_s}$ and we are given access to a stochastic gradient oracle with $\eps$-tail.

    Then, as long as the input $x_0$ satisfies $\nrm{\nabla f(x_0)}\leq G$, the approximate proximal oracle with $\epsprox=10(m_s+M)$ can be implemented with probability at least $1-\eps_n(M)$ using $O(n\log(G/(m_s+M)))$ queries to the stochastic gradient oracle. 
\end{lemma}

\paragraph{Notation}
For any function $f:\RR^d\to\RR$ such that $Z_f\deq \int_{\RR^d} e^{-f(x)}dx<+\infty$, we define $\mu_f$ to be the distribution over $\RR^d$ with density $\mu_f(x)=\frac{1}{Z_f}e^{-f(x)}$.

For $B > 0$, we write $\pclip{\cdot} \deq \max\{-B, \min\{B, \cdot\}\}$ \scedit{and $\trunc{\cdot} \deq (\abs\cdot-B)_+$}. We use $\leqsim$ and $O(\cdot)$ to hide absolute constants, i.e., $f\leqsim g$ (and $f=O(g)$) if there is an absolute constant such that $f\leq Cg$.
The notation $\wt O(\cdot)$ hides logarithmic factors.

\section{High-accuracy sampling with stochastic queries}\label{sec:upper}

We build up to our results in three steps.
Our methods build upon first-order rejection sampling (FORS), a meta-algorithm recently developed in~\citet{chen2026high} which simulates rejection sampling given unbiased estimators of the log-density ratio between the proposal and target.
Therefore, we first review the FORS framework in \cref{sec:motivation}.
Then, in \cref{ssec:gaussian-tilt-stoch}, we instantiate FORS for the problem of sampling from a Gaussian tilt distribution, thereby showing that the results of~\citet{chen2026high} are robust to stochastic gradient noise.
Finally, in \cref{ssec:log-concave}, we combine the results of \cref{ssec:gaussian-tilt-stoch} with the proximal sampler algorithm~\citep{LeeSheTia21RGO, Chen+22ProxSampler} to establish our main results for log-concave sampling.

\subsection{Background on first-order rejection sampling (FORS)}\label{sec:motivation}

To motivate the FORS algorithm, we replicate the motivating example of~\citet{chen2026high} here. Consider the simple problem of sampling from a density $p\propto e^{-f}$, where $f : [0,1]\to\R$, $f(0) = 0$, and $-1 \le f' \le 1$.
In order to perform rejection sampling with the base measure $\unif([0,1])$, we must generate randomness $b\sim \Ber(ce^{-f(x)})$ for any given $x\in[0,1]$.
To do so, one typically assumes access to evaluations of $f$ itself.
The novelty of FORS lies in recognizing that this is unnecessary---it suffices to produce \emph{unbiased estimators} of $f$.

A more general version of this idea is known as the ``Bernoulli factory'' problem \citep{keane1994bernoulli,nacu2005fast}, and variants of this idea can be found in multiple domains~\citep[e.g.,][]{Wag1988SDE, Pap11Diffusion}.
It can be stated as the following abstract task:
\begin{center}
    \textbf{Task:} Given i.i.d.\ random variables $W_1,W_2,W_3,\dotsc$ in $[-1,1]$, generate a sample $b\sim \Ber(ce^{\En W_1})$.
\end{center}

To solve this, write the Taylor series as
\begin{align}\notag
    e^{\En W_1}=e^{-1}\cdot e^{\En[1+W_1]}
    =\sum_{j\geq 0} \frac{e^{-1}}{j!}\, \bigl(\En[1+W_1]\bigr)^j\,.
\end{align}
Suppose that $J\sim \Poi(2)$ is independent of the i.i.d.\ sequence $W_1,W_2,W_3,\dotsc$. Then we notice that
\begin{align}\notag
    e^{\En W_1}=e\En\brk[\Big]{\prod_{j=1}^J \bigl(\frac{1+W_j}{2}\bigr)}\,,
\end{align}
so we can set $b\sim \Ber\prn[\big]{\prod_{j=1}^J \bigl(\frac{1+W_j}{2}\bigr)}$.
Indeed, $\Pr(b=1) = \E\prod_{j=1}^J \bigl(\frac{1+W_j}{2}\bigr) = e^{-1+\En W_1}$.

In summary, to generate a sample $b\sim \Ber\prn{ce^{-f(x)}}$, it suffices to have access to (a random number of) unbiased estimates of $f(x)$. In~\citet{chen2026high}, this was leveraged to produce high-accuracy samplers which only use queries to the \emph{derivative} $f'$, via the representation $f(x)=\En_{y\sim \unif([0,x])}[xf'(y)]$.
In our work, our goal is to leverage this phenomenon in order to produce high-accuracy samplers that \emph{tolerate stochasticity in the gradient oracle}.

We are now ready to state the general FORS meta-algorithm. Given a proposal distribution $q$, a tilt function $w$, and a tuneable parameter $B = \Theta(1)$, the goal of \cref{alg:fors} is to produce a sample from $\widehat p(x)\propto q(x)\, e^{w(x)}$ without having access to the \emph{value} $w(x)$. Instead, for each $x\in\RR^d$, we can generate i.i.d.\ samples $W_1,W_2, W_3\dotsc$ such that $\En[W_1\mid x]=w(x)$. 
Let $\mathcal W_x$ denote the conditional distribution of $W_1$ given $x$.

\begin{algorithm}
\caption{First-order rejection sampling (FORS)}\label{alg:fors}
\begin{algorithmic}
\STATE \textbf{Input:} Parameter $B > 0$, proposal distribution $q$ over $\R^d$, estimator distributions $(\cW_x)_{x\in\R^d}$ supported on $[-B,B]$
\FOR{$i=1,2,3,\dotsc$}
\STATE Sample $x\sim q$.
\STATE Sample $J\sim \Poi(2\B)$.
\STATE Sample i.i.d. $W_1,\dotsc,W_J \sim \mathcal W_x$.
\STATE Output $x$ with probability $\prod_{j=1}^J \frac{B+W_j}{2B}$.
\ENDFOR
\end{algorithmic}
\end{algorithm}

\begin{theorem}[{FORS guarantee,~\citet[Theorem 3.1]{chen2026high}}]\label{thm:fors}
    \cref{alg:fors} outputs a random point with density $\widehat p(x)\propto q(x)\, e^{\En[W_1\mid x]}$.
    The number of sampled $W_j$'s is bounded, with probability at least $1-\delta$, by $3Be^{2B}\log(2/\delta)$.

    Moreover, if~\cref{alg:fors} is called $T$ times, then with probability at least $1-\delta$, the total number of sampled $W_j$'s is $O(Be^{2B}\,(T+\log(1/\delta)))$.
\end{theorem}

\subsection{Sampling from Gaussian tilts with stochastic queries}\label{ssec:gaussian-tilt-stoch}

The goal of this section is to sample from the following Gaussian tilt distribution:
\begin{align}\label{eq:gaussian-tilt}
    \nu(x) \propto \exp\prn[\Big]{-f(x) - \frac{1}{2\eta}\,\nrm{x-x_0}^2}\,.
\end{align}
In \cref{ssec:log-concave}, this will be used as a subroutine for the proximal sampler algorithm~\citep{LeeSheTia21RGO, Chen+22ProxSampler}, leading to new guarantees for log-concave sampling.

\begin{remark}[Diffusion models]
    Leveraging the fact that the backward transition kernels along a diffusion model are also Gaussian tilts of the form~\eqref{eq:gaussian-tilt},~\citet{chen2026high} used FORS to provide the first \emph{high-accuracy} sampling guarantees for diffusion models under minimal data assumptions.
    Similarly, the results we present below could also be applied to that setting to show that diffusion sampling can be made \emph{robust to stochastic errors in the score evaluations}.
    For brevity, we do not pursue this application here.
\end{remark}

We now consider instantiating FORS for the Gaussian tilt distribution~\eqref{eq:gaussian-tilt}.
Let $\gam{\lr}\ldef \gamma(x;z,r)$ be any \emph{path function} such that $\gam{1}=x$ and $\gam{0}=\wb{\gamma}(z)$ is independent of $x$; here, $z\sim P$ is an external source of randomness.
Then, by the fundamental theorem of calculus,
\begin{align*}
    f(x) - \E_{z\sim P} f(\wb{\gamma}(z)) = \E_{r\sim\Unif([0,1]),\,z\sim P}\langle \dot\gamma_{z,r}(x), \nabla f(\gamma_{z,r}(x))\rangle\,.
\end{align*}
If we choose the proposal $q=\normal{\wh x,\eta\Id}$, where $\wh x$ is a fixed base point chosen so that $q\approx \nu$ (made precise in \cref{thm:gaussian_tilt_stoc}), then \colt{$q(x) \propto \exp\bigl(-\frac{1}{2\eta}\norm{x-\wh x}^2\bigr)$,}
\arxiv{
\begin{align*}
    q(x) \propto \exp\Bigl(-\frac{1}{2\eta}\norm{x-\wh x}^2\Bigr)\,,
\end{align*}
}
and hence
\begin{align*}
    \log \nu(x)-\log q(x)
&= \scedit{-}f(x) - \frac{1}{2\eta}\,\|x-x_0\|^2 + \frac{1}{2\eta}\,\|x-\xhat\|^2 + \const\\
&= \frac{1}{\eta}\,\langle x_0-\xhat,x \rangle -f(x)+ \const\,.
\end{align*}
Thus, applying the path integral formula to $h(x)=\eta^{-1}\,\langle x_0-\xhat,x\rangle -f(x)$, we can express
\begin{align*}
    \log \rgo(x)-\log q(x)%
    = \E_{r\sim\Unif([0,1]),\,z\sim P}\Bigl\langle \dot\gamma_{z,r}(x),\, \eta^{-1}(x_0-\wh x)  -  \nabla f(\gamma_{z,r}(x))\Bigr\rangle+\const\,.
\end{align*}
By the guarantee of \cref{thm:fors}, it suggests that we use the unbiased estimator $W_{r,z,x} \deq \langle \dot\gamma_{z,r}(x), u - \nabla f(\gamma_{z,r}(x))\rangle$, with $u \deq (x_0-\wh x)/\eta$.
Actually, since the $W$'s in \cref{alg:fors} must lie in $[-B,B]$, we truncate the estimator to lie in this range.
Further, we replace the exact gradients by stochastic gradients, leading to
\begin{align*}
    \wh W_{r,z,g,x}
    \deq \pclip{\langle \dot\gamma_{z,r}(x), u - g\rangle}\,, \qquad g\sim \Ograd{\gamma_{z,r}(x)}\,.
\end{align*}
Below, we choose the base point of the proposal $\wh x$, the path function $\gamma_{z,r}$, and the noise distribution $P$ in order to optimize the dimension dependence of our result.

\begin{theorem}[Sampling from Gaussian tilts]\label{thm:gaussian_tilt_stoc}
    Suppose that \cref{ass:holder} holds, $\Ograd{\cdot}$ has $\eps$-tail, $n\geq 1$, and $B=\Theta(1)$.

    Instantiate~\cref{alg:fors} as follows: 
    \begin{itemize}
        \item $q = \normal{\xhat,\eta\Id}$, where $\xhat$ is drawn from $\Oprox{\xz}$. We write $u\deq \frac{\xz-\xhat}{\eta}$.
        \item $\cW_x$ is the law of $\pclip{W_{r,z,g,x}}$, where
        \begin{align}\label{eq:def-stoc-W}
        W_{r,z,g,x}=\tri{\gamp{\lr}, u-g}\,,\quad
        r\sim \Unif([0,1])\,,\; z\sim \normal{0,\eta\Id}\,,\; 
        g\sim \Ogradn{\gam{\lr}}\,,
        \end{align}
        and
        \begin{align}\label{eq:gamma}
            \gamz[x][z][\lr]=\alr x + (1-\alr) \xhat + \blr z\,, \qquad
            \alr=\sin(\pi\lr/2)\,, ~~\blr=\cos(\pi\lr/2)\,,
        \end{align}
        so that $\gamzp[x][z][\lr]=\alrp (x -\xhat) + \blrp z$. 
    \end{itemize}
    Then, conditioned on $\nrm{u-\nabla f(\xhat)}^2\leq \epsprox^2$ and
    \begin{align*}
        \eta^{-1} \gg \prn[\Big]{ \beta_s^2d^s\log(1/\delta)+\frac{s\beta_s^2}{d^{1-s}}\log^2(1/\delta) }^{1/(1+s)}+(M^2+\epsprox^2)\log(1/\delta)\,,
    \end{align*}
    the law $\wh\nu$ of~\cref{alg:fors} satisfies $\Dtv{\nu}{\wh\nu} \le \delta+C\En_{x\sim \nu}\min\crl*{\eps_n(M;x),1}$, where $C$ is an absolute constant.
\end{theorem}

In our application to log-concave sampling, $\eta$ will be interpreted as a step size, and hence the overall complexity of sampling will scale with $\eta^{-1}$, multiplied by the batch size $n$ and other distribution-specific pre-factors.
We pause to give several remarks to elucidate the dependencies in this result.

\begin{remark}[Dimension dependence]
    The first term requires taking an inverse step size $\eta^{-1} \gg \beta_0^2$ in the Lipschitz case ($s=0$), and $\eta^{-1} \gg \beta_1 d^{1/2}$ in the smooth case ($s=1$).
    This matches state-of-the-art results for high-accuracy sampling~\citep{FanYuaChe23ImprovedProx, AltChe24Warm, chen2026high}, except that we allow for stochastic gradient queries.
\end{remark}

\begin{remark}[Proximal tolerance]
    Since the theorem already requires taking $\eta^{-1} \gg \beta_0^2 + M^2$ in the Lipschitz case, and $\eta^{-1} \gg M^2$ in the smooth case, then \cref{lem:implement-prox} (with $n=1$) implies that implementing the approximate proximal oracle with the stochastic gradient oracle only incurs a \emph{logarithmic overhead}.
\end{remark}

\begin{remark}[Accuracy dependence]
    To reach a final error of $\delta$, we need to take $\eta^{-1}$ at least of order $M^2 = \epsilon_n^{-1}(\delta)^2$.
    We elucidate this in two cases of particular interest.
    \begin{itemize}
        \item (Sub-Gaussian tails) Suppose that the stochastic gradients have sub-Gaussian tails, which corresponds to $\zeta = 2$ in \cref{example:sub-poly}.
        Then, we can take $n=1$ and $M^2 \asymp \gvar^2 \log(1/\delta)$, thus the final term requires $\eta^{-1} \gg \gvar^2 \log^2(1/\delta)$.
        Hence, this leads to a \emph{high-accuracy} guarantee.
        \item (Bounded variance) Suppose now that the stochastic gradients merely have variance bounded by $\sigma^2$.
        Since $\epsilon_1(M) \lesssim \sigma^2/M^2$, we can choose $M^2 \asymp \sigma^2/\delta$.
        Thus, the dependence on $\delta$ becomes $\eta^{-1} \gg \delta^{-1} \log(1/\delta)$.
        Although it suffices to take $n=1$ here, to avoid error accumulation in the next section we will eventually have to apply batching ($n > 1$). After doing so, the iteration complexity remains $1/\delta$ (up to logarithmic factors).

        In \cref{sec:lower}, we will show that the $1/\delta$ rate is in fact \emph{optimal} under the bounded variance assumption (\cref{prop:Ograd-lower}). Thus, high-accuracy sampling requires light-tailed stochastic gradients.
    \end{itemize}
\end{remark}

Parallel to \cref{thm:gaussian_tilt_stoc}, we show that it is also possible to sample from the Gaussian tilt distribution with only stochastic value queries, provided that the error of the stochastic value oracle is sufficiently small.

\begin{theorem}\label{thm:gaussian_tilt_stoc_0th}
   Suppose that \cref{ass:holder} holds, and $\Oeval{\cdot}$ has $\eps$-tail. Suppose that $n\geq1$ and $B=\Theta(1)$.

   Instantiate~\cref{alg:fors} as follows: 
   \begin{itemize}
       \item $q = \normal{\xhat,\eta\Id}$, where $\xhat$ is drawn from $\Oprox{\xz}$. We write $u=\frac{\xz-\xhat}{\eta}$. 
       \item $\cW_x$ is the law of $\pclip{W_{z,v,v',x}}$, where
       \begin{align}\label{eq:def-stoc-W-0th}
       W_{z,v,v',x}=v'-v-\tri{u,x-z}\,,\qquad
       z\sim q\,,\; 
       v\sim \Oevaln{x}\,,\; v'\sim \Oevaln{z}\,.
       \end{align}
   \end{itemize}
   Then, conditioned on $\nrm{u-\nabla f(\xhat)}\leq \epsprox$, the law $\wh\nu$ of~\cref{alg:fors} satisfies $\Dtv{\nu}{\wh\nu} \le \delta+C\En_{x\sim \nu}\min\crl*{\eps_n(\scedit{B/4};x),1}$, provided that
   \begin{align*}
       \eta^{-1} \gg \prn[\Big]{ \beta_s^2d^s\log(1/\delta)+\frac{s\beta_s^2}{d^{1-s}}\log^2(1/\delta) }^{1/(1+s)}+\epsprox^2\log(1/\delta).
   \end{align*}
\end{theorem}

We note that in general, implementing the proximal oracle with only noisy queries will incur additional computational cost. However, with the choice of $\xhat=\xz$, it is trivially guaranteed that $\epsprox=\nrm{\nabla f(\xz)}$.
In this case, as we will see in \cref{thm:prox_0th} below, the term $\epsprox^2$ in fact dominates the complexity.

\subsection{Log-concave sampling}\label{ssec:log-concave}

To apply our results to log-concave sampling (and beyond), we apply the results of the previous section to the proximal sampler algorithm~\citep{LeeSheTia21RGO, Chen+22ProxSampler}.
Given a target distribution $\mu \propto e^{-f}$, the proximal sampler aims to sample from the augmented distribution
\begin{align*}
    \bar\pi(x,y) \propto \exp\prn[\Big]{-f(x) -\frac{1}{2\eta}\,\nrm{y-x}^2}\,.
\end{align*}
It does so by applying Gibbs sampling to $\bar\pi$.
Concretely, for $n=0,1,2,\dotsc$ and an initial point $X_0 \sim \mu_0$, repeat:
\begin{enumerate}
    \item Sample $Y_n \sim \bar \pi^{Y|X=X_n} = \normal{X_n,\eta I}$.
    \item Sample $X_{n+1} \sim \bar \pi^{X|Y=Y_{n}}$.
\end{enumerate}
The distribution $\bar\pi^{X|Y=y}$ is known as the \emph{restricted Gaussian oracle} (RGO), and it is exactly the Gaussian tilt distribution~\eqref{eq:gaussian-tilt} with $x_0=y$.
We therefore combine our result in \cref{thm:gaussian_tilt_stoc} for implementing the RGO, together with existing results on the convergence of the proximal sampler itself, to deduce the following sampling corollaries.
We begin by recalling the definitions of functional inequalities.

\begin{definition}[Poincar\'e]
    A distribution $\pi$ satisfies a \emph{Poincar\'e inequality (PI)} with constant $C$ if for all compactly supported and smooth test functions $h : \R^d\to\R$,
    \begin{align*}
        \Var_{X\sim \pi}(h(X)) \le C\En_{X\sim\pi}[\nrm{\nabla h(X)}^2]\,.
    \end{align*}
    We let $\CPI(\pi)$ be the smallest constant $C$ such that $\pi$ satisfies PI with constant $C$.
\end{definition}

\begin{definition}[Log-Sobolev]
    A distribution $\pi$ satisfies a \emph{log-Sobolev inequality (LSI)} with constant $C$ if for all compactly supported and smooth test functions $h : \R^d \to \R$,
    \begin{align*}
        \mathrm{Ent}_{X\sim \pi}(h^2(X))
        \deq \E_{X\sim \pi}\brk[\Big]{h^2(X) \log \frac{h^2(X)}{\E_{X\sim \pi}[h^2(X)]}} \le 2C\En_{X\sim\pi}[\nrm{\nabla h(X)}^2]\,.
    \end{align*}
    We let $\CLSI(\pi)$ be the smallest constant $C$ such that $\pi$ satisfies LSI with constant $C$.
\end{definition}

It is well-known~\citep[e.g.,][]{BGL14} that if $\pi$ is $\alpha$-strongly log-concave (SLC), i.e., $-\log\pi$ is $\alpha$-strongly convex, then it satisfies LSI with constant $1/\alpha$, and if $\pi$ satisfies LSI with constant $1/\alpha$, then it satisfies PI with constant $1/\alpha$.
These represent meaningful enlargements of the class of SLC measures which still allow for tractable sampling. For example, unlike SLC, LSI is robust to bounded perturbations of the log-density; and unlike LSI\@, PI allows for capturing measures without sub-Gaussian tails (e.g., the two-sided exponential).
See~\citet{Chewi26Book} for further background in the context of sampling.

We now present a suite of results by combining \cref{thm:gaussian_tilt_stoc} (\cref{thm:gaussian_tilt_stoc_0th}) with the guarantees of the proximal sampler~\citep{Chen+22ProxSampler}. 
Let $\phi_M(\delta) \deq \inf\{n\ge 1 : \epsilon_n(M) \le \delta/(10C)\}$. We note that this can be relaxed to the ``in-distribution error'':
\begin{align*}
    \phi_{M,N}(\delta) \deq \inf\crl[\bigg]{n\ge 1 : \frac{1}{N}\sum_{k=1}^N \En_{x\sim \mu_k}\min\crl{\epsilon_n(M;x), 1} \le \delta/(10C)}\,,
\end{align*}
where  $\mu_k$ is the distribution of the $X_k$ in the exact proximal sampler.

\begin{theorem}\label{thm:prox_1st}
    Suppose that \cref{ass:holder} holds \scedit{for some $s \in [0,1]$}, and that $\Ograd{\cdot}$ has $\epsilon$-tail. Suppose that we are given an initial distribution $\mu_0$ such that $\log(1+\Dchis{\mu_0}{\mu})\leq \Delta$. 

    Choose
    \begin{align}\label{eq:eta-1st}
        \frac{1}{C\eta} = \prn[\big]{\beta_s^2 d^s\log(N/\delta) + \beta_s^2\log^2(N/\delta)}^{1/(1+s)}+M^2\log(N/\delta)
    \end{align}
    for a sufficiently large universal constant $C > 0$.
    Let $\wh\mu$ denote the law of the output of the proximal sampler initialized at $\mu_0$, where in each step the RGO is implemented by \cref{thm:gaussian_tilt_stoc}.
    Then, the proximal sampler ensures $\Dtv{\wh\mu}{\mu}\leq \delta$ using at most $N\phi_M(\delta/N) \log A$ queries to $\Ograd{\cdot}$ in the following situations. 

    \begin{enumerate}
        \item Suppose that $\mu$ satisfies a log-Sobolev inequality with constant $\CLSI(\mu) < \infty$ and $s = 1$ (i.e., $f$ is smooth).
        Then, 
        \begin{align*}
            N \lesssim \CLSI(\mu)\,\prn[\big]{ \beta_1 d^{1/2} \log^{3/2} A + (\beta_1+M^2) \log^2A }\,,
        \end{align*}
        where $A \deq d + \Delta + \delta^{-1}+\CLSI(\mu)\,(\beta_1+M^2)$.
        \item Suppose that $\mu$ satisfies a Poincar\'e inequality with constant $\CPI(\mu) < \infty$.
        Then, 
        \begin{align*}
            N \lesssim \CPI(\mu)\,\prn[\big]{ \prn{\beta_s^2d^s\log A+\beta_s^2\log^2 A}^{1/(1+s)}+M^2\log A }\,(\Delta+\log(1/\delta))\,,
        \end{align*}
        where $A \deq d + \Delta + \delta^{-1}+\CPI(\mu)\,(\beta_s^{2/(1+s)}+M^2)$.
        \item Suppose that $\mu$ is log-concave.
        Then,
        \begin{align*}
            N \lesssim \prn[\big]{\prn[\big]{\beta_s^2d^s\log A+\beta_s^2\log^2 A}^{1/(1+s)}+M^2\log A}\cdot \frac{W_2^2(\mu_0,\mu)}{\delta^2}\,,
        \end{align*}
        where $A \deq d + \Delta + \delta^{-1} + (\beta_s^{2/(1+s)} + M^2)\, W_2^2(\mu_0,\mu)$.
    \end{enumerate}
\end{theorem}

\begin{remark}[Dependence on $\delta$]\label{rmk:delta-rate}
    As an illustration, we describe the implied query complexity in the following special cases.
    \begin{itemize}
        \item If the stochastic gradients have subexponential tails (\cref{example:sub-poly} with $\zeta=1$), then for $M \ge \gvar$, we can take $\phi_M(\delta) \lesssim (\gvar/M)\log(1/\delta)$.
        Therefore, we can take $M \asymp \gvar$ in all of the results above, and the iteration complexity equals $\wt{O}(N)$. 
        For example, in the smooth LSI case, the query complexity reads $\wt{O}(\CLSI(\beta_1 d^{1/2}+\gvar^2))$.
        \item On the other hand, in the bounded variance case (\cref{example:poly} with $k=1$), we can take $\phi_M(\delta) \lesssim \sigma_2^2/(\delta M^2)$, and the total query complexity becomes $(N\vee N^2 \sigma_2^2/(\delta M^2)) \log A$.
        We then choose $M$ to balance the terms. For example, in the smooth LSI case, the query complexity reads
        \begin{align*}
            N\log A + \CLSI \brk[\big]{(\beta_1^2 d/M^2) \log^3A + M^2\log^4A} \,(\sigma_2^2 \log A)/\delta\,,
        \end{align*}
        This leads to a total query complexity of $\wt O(\kappa d^{1/2}\, (1+ \CLSI \sigma_2^2/\delta))$, where $\kappa=\CLSI \beta_1$ is the condition number.
    \end{itemize}
\end{remark}

We emphasize that while sampling guarantees with stochastic gradients are well-studied~\citep[e.g.,][]{Dal17Analogy, DalKar19Friendly, DurMajMia19LMCCvxOpt, Bal+22NonLogConcave, Huang+24SGProx, LuYeZho25StochGrad}, our contribution is to provide \emph{high-accuracy} guarantees, provided that the stochastic gradients have light tails.

In the next section, we show that the assumption on the tails of the stochastic gradient is necessary.

We also provide a corresponding result for stochastic value queries.

\begin{theorem}\label{thm:prox_0th}
    Suppose that \cref{ass:holder} holds and that $\Oeval{\cdot}$ has $\epsilon$-tail. Suppose that we are given an initial distribution $\mu_0$ such that $\log(1+\Dchis{\mu_0}{\mu})\leq \Delta$. 

    Choose
    \begin{align}\label{eq:eta-0th}
        \frac{1}{C\eta} = (\beta_sd^s)^{2/(1+s)}\,\prn[\Big]{1+\frac{\Delta+\log(N/\delta)}{d}}\log(N/\delta)
    \end{align}
    for a sufficiently large universal constant $C > 0$.
    Let $\wh\mu$ denote the law of the output of the proximal sampler initialized at $\mu_0$, where in each step the RGO is implemented by \scedit{\cref{thm:gaussian_tilt_stoc_0th}}.
    Then, the proximal sampler ensures $\Dtv{\wh\mu}{\mu}\leq \delta$ using at most $N\phi_{1}(\delta/(4N))$ queries to $\Oeval{\cdot}$ in the following situations. 

    \begin{enumerate}
        \item Suppose that $\mu$ satisfies a log-Sobolev inequality with constant $\CLSI(\mu) < \infty$ and $s = 1$ (i.e., $f$ is smooth).
        Then, 
        \begin{align*}
            N \lesssim \CLSI(\mu)\,\beta_1\, (d+\Delta+\log A)\log^2 A\,,
        \end{align*}
        where $A \deq d + \Delta + \delta^{-1}+\CLSI(\mu)\,\beta_1$.
        \item Suppose that $\mu$ satisfies a Poincar\'e inequality with constant $\CPI(\mu) < \infty$.
        Then, 
        \begin{align*}
            N \lesssim \CPI(\mu)\,(\beta_sd^s)^{2/(1+s)}\,\prn[\Big]{1+\frac{\Delta+\log A}{d}}\,(\Delta+\log(1/\delta))\log A\,,
        \end{align*}
        where $A \deq d + \Delta + \delta^{-1}+\CPI(\mu)\,\beta_s^{2/(1+s)}$.
        \item Suppose that $\mu$ is log-concave.
        Then,
        \begin{align*}
            N \lesssim (\beta_sd^s)^{2/(1+s)}\,\prn[\Big]{1+\frac{\Delta+\log A}{d}}\log A\cdot \frac{W_2^2(\mu_0,\mu)}{\delta^2}\,,
        \end{align*}
        where $A \deq d + \Delta + \delta^{-1} + \beta_s^{2/(1+s)}  W_2^2(\mu_0,\mu)$.
    \end{enumerate}
\end{theorem}

Note that under noisy value queries of sub-Gaussian tail (\cref{example:sub-poly}), it holds that $\eps_n(1)\leq e^{-n/\gvar^2}$ for $n\gg \gvar^2$, and hence $\phi_1(\delta)=O(\gvar^2\log(1/\delta)+1)$. 
Thus, assuming that $f$ is $\alpha$-strongly convex and $\beta$-smooth, the query complexity (roughly) scales as $\wt O(\kappa d\cdot \fcedit{\max\{\gvar^2,1\}})$, where $\kappa=\beta/\alpha$ is the condition number.
By reduction to zeroth-order optimization, it is expected that in this setting, sublinear dependence on $d$ cannot be achieved.

\subsection{Application: finite-sum sampling}\label{ssec:finite_sum}

\newcommand{\prox}{\mathsf{prox}}
\newcommand{\Xbar}{\wb{X}}
\newcommand{\Xhat}{\wh{X}}
\newcommand{\nubar}{\wb{\nu}}

We now consider a slightly more abstract formulation of the finite-sum sampling problem. For simplicity, we focus on the smooth setting.

\begin{assumption}\label{asmp:smooth-stoc}
The function $f$ takes the form $f(x)=\En_{w\sim P} F(x;w)$, and computing $\nabla F(x;w)$ requires unit cost. Furthermore, $\nrm{\nabla F(x;w)-\nabla F(x';w)}\leq \beta_1\,\nrm{x-x'}$ for all $w$ and $x,x'\in\RR^d$.
\end{assumption}

\begin{theorem}\label{thm:finite-sum}
Suppose that \cref{asmp:smooth-stoc} holds and that the initial distribution $\mu_0$ satisfies $\log(1+\Dchis{\mu_0}{\mu})\leq \Delta$. Consider implementing the proximal sampler as follows.

\begin{itemize}
    \item Initialize $X_0\sim \mu_0$.
    \item For each $k\geq 0$, sample $Y_k\sim \normal{X_k,\eta\Id}$.
    \item If $k\bmod K = 0$, query $\Xhat_{k+1} \sim \Oprox{Y_k}$ and compute $\nabla f(X_k)$. Otherwise, set $\Xhat_{k+1}\deq \Xhat_{m(k)\scedit{+1}}+Y_k-Y_{m(k)}$, where $m(k)\deq K\floor{k/K}$.
    \item Let $O_{k+1}(x)$ denote the distribution of $\nabla F(x;w)-\nabla F(X_{m(k)};w)+\nabla f(X_{m(k)})$ under $w\sim P$.
    \item Instantiate \cref{thm:gaussian_tilt_stoc} with oracle $O_{k+1}$ and center $\Xhat_{k+1}$ to generate $X_{k+1}$.
\end{itemize}
Then, assuming that each call to the proximal oracle $\Oprox{\cdot}$ succeeds with probability at least $1-\delta$ and
\begin{align}\label{eq:finite-sum-eta}
    \frac{1}{\beta_1\eta}\gg \sqrt{Kd}+K^{2/3}(d+\Delta+\log(K/\delta))^{1/3}+(\epsprox^2/\beta_1+1)\log(K/\delta),
\end{align}
the distribution $\muhat_N$ of our algorithm satisfies $\Dtv{\mu_N}{\muhat_N}\leq N\delta$.
\end{theorem}

Note that when $P=\unif([m])$, i.e., $f(x)=\frac{1}{m}\sum_{i=1}^m f_i(x)$, evaluating $\nabla f(x)$ has cost $m$, and the proximal oracle can be implemented via standard SVRG methods in $\wt{O}(m)$ time. Therefore, assuming $\Delta=\wt{O}(d)$, we can choose $K=m$ to obtain a query complexity of $\wt{O}(m+\kappa\,(\sqrt{md}+m^{2/3}d^{1/3}))$, where $\kappa\deq \CLSI(\mu)\beta_1$.
\scedit{This improves upon the $\wt O(m + \kappa\,(\sqrt{md} + d))$ complexity achieved in~\citet{LeeSheTia21RGO} in the regime $m\le d$. Moreover, by combining the two results, i.e., using our result for $m\le d$ and their result for $m > d$, it yields an overall bound of $\widetilde O(m + \kappa\sqrt{md})$.}

\section{Lower bound: light tails are necessary for high-accuracy sampling}\label{sec:lower}

\subsection{A simple lower bound}

We establish lower bounds for sampling with stochastic gradient queries under oracles with bounded $\psi$-moment.
\begin{definition}
Let $\psi:[0,+\infty)\to [0,+\infty)$ be an increasing function such that $\psi(0)=0$. An oracle $\Ograd{\cdot}$ is a $\psi$-oracle for $f$ if for any $x\in\RR^d$, under $g\sim \Ograd{x}$, it holds that $\En[g]=\nabla f(x)$ and $\En\psi(\nrm{g-\nabla f(x)})\leq 1$.
\end{definition}

In the following, we present a simple information-theoretic argument based on the goal of sampling from a one-dimensional Gaussian $p_\theta=\normal{\theta,\tfrac{1}{\alpha}\Id}$. Here $p_\theta(x)\propto \exp(-f_\theta(x))$ with $f_\theta(x)=\frac{\alpha}{2}(x-\theta)^2$ and $\nabla f_\theta(x)=\alpha(x-\theta)$. Since $f_\theta$ is $\alpha$-strongly convex, the LSI holds with $\CLSI(p_\theta)\leq \tfrac1\alpha$.

\begin{proposition}[Lower bound]\label{prop:Ograd-lower}
Fix any increasing function $\psi:[0,+\infty)\to [0,+\infty)$ such that $\psi(0)=0$. Suppose that $T\geq 1$ and $\delta\in(0,1]$, and there is an algorithm $\Alg$ such that for any $\theta\in\crl{0,\scedit{\delta/\sqrt\alpha}}$, given any $\psi$-oracle $O$ for $f_\theta$, return a sample $x\sim \Alg(O)$ using $T$ queries to $O$ and $\Dtv{p_\theta}{\Alg(O)}\leq \frac{\delta}{10}$. Then, it holds that
\begin{align*}
    T\geq \frac{1}{10\sqrt{\alpha}}F_\psi(\sqrt{\alpha}\delta)\,, \qquad F_\psi(\theta)\deq \sup\crl{u\geq \theta: (1-\psi(\theta))\cdot u\geq \theta \cdot \psi(u)}\,.
\end{align*}
\end{proposition}

\begin{remark}\label{rmk:delta-rate-lower}
As demonstration, we describe the implied query complexity lower bound in the following special cases.

\begin{itemize}
\item Consider the case where the only assumption on $\Ograd{\cdot}$ is that the variance is bounded by $\sigma^2$. Then we can take $\psi(m)=\frac{m^2}{\sigma^2}$ and $F_\psi(\theta)=\frac{\sigma^2}{\theta} - \theta$ for any $\theta\leq \frac{\sigma}{2}$, i.e., $\Omega(\sigma^2/(\alpha\delta))$ queries are necessary for stochastic gradients with only bounded second moment. 
Thus, in this case, our upper bounds are optimal (\cref{rmk:delta-rate}), at least with respect to the dependence on $\delta$.
\item More generally, for $\psi(m)=(m/\sigma)^s$ with $s>1$, we have $F_\psi(\delta)\asymp \frac{\sigma^{s/(s-1)}}{\delta^{1/(s-1)}}$, i.e., if the only assumption on $\Ograd{\cdot}$ is a bounded $s$-th moment, \cref{prop:Ograd-lower} yields a lower bound of $\Omega\prn[\big]{\frac{(\sigma/\sqrt{\alpha})^{s/(s-1)}}{\delta^{1/(s-1)}}}$ queries. In particular, taking $s\to 1$ implies that it is intractable to sample with stochastic gradients with only bounded first moment.
\item For $\Ograd{\cdot}$ with sub-exponential tail, we take $\psi(m)=e^{(m/\sigma)^\zeta}-1$ with $\zeta>0$. Then $F_\psi(\delta)\geq \Omega(\sigma\log^{1/\zeta}(\sigma/\delta))$, and \cref{prop:Ograd-lower} yields a lower bound of $\Omega\prn[\big]{\frac{\sigma}{\sqrt{\alpha}}\log^{1/\zeta}\frac{\sigma}{\sqrt\alpha\delta}}$.
\scedit{On the other hand, in this case, the argument of~\citet{ChaBarLon22Lower} yields an alternate lower bound of $\Omega(\frac{\sigma^2}{\alpha})$ in this case.}
\end{itemize}
\end{remark}

\begin{remark}[Dimensional dependence]
For light-tailed (e.g., sub-exponential) stochastic gradients, our upper bound scales as $\widetilde O(\kappa d^{1/2} + \sigma^2/\alpha)$, whereas our lower bound does not scale with $d$. The $\kappa d^{1/2}$ term is therefore not captured by the lower bound; since it is independent of the variance proxy $\sigma$, it reflects the baseline cost of sampling even with an exact oracle. It may be possible to reduce this term, as the best-known lower bound for exact-oracle sampling scales only as $\min\{\sqrt{\kappa}, d\}$~\citep{Chewi+23QueryLower}. However, closing this gap remains a long-standing open question.
\end{remark}

\subsection{Revisiting the lower bound of~\citet{ChaBarLon22Lower}}\label{ssec:bartlett}

Here, we discuss the lower bound of~\citet{ChaBarLon22Lower}, which also applies to sampling with stochastic gradients, in order to avoid potential misunderstandings.

Their main lower bound shows that there is a strongly log-concave and log-smooth distribution, with condition number $\kappa = O(1)$, such that it requires $\Omega(\sigma^2/\delta^2)$ queries to reach $\delta$ error in TV distance.
Here, $\sigma^2$ is the variance of the stochastic gradients.\footnote{A variance bound of $\sigma^2$ in our convention corresponds to a variance bound of $\sigma^2 d$ in theirs.}
This appears to contradict our upper bound, which only requires $O(1/\delta)$ queries in the bounded variance case.
Moreover, inspection of their lower bound reveals that it holds when the stochastic gradient oracle is produced by adding Gaussian noise; in particular, the stochastic gradients have sub-Gaussian tails.
In such a setting, we have produced algorithms whose complexity scales as $O(\polylog(1/\delta))$.

To resolve this apparent contradiction, we remark that~\citet[Theorem 4.1]{ChaBarLon22Lower} requires taking the strong log-concavity parameter $\alpha \lesssim \delta^2$.
Our upper bounds, which generally incur a dependence of $\sigma^2/\alpha$, therefore match their $\Omega(\sigma^2/\delta^2)$ lower bound for their hard examples up to logarithmic terms.
However, their lower bounds do not address the question of what the best dependence on $\delta$ is, provided that $\alpha$ is bounded away from zero.
This is the reason why we proved \cref{prop:Ograd-lower}.

We leave it as an open question to prove a more general lower bound which captures the dependence, not just on $\delta$, but on other problem parameters such as $\kappa$ and $d$.
In the case of exact oracle access, proving lower bounds for sampling remains notoriously challenging, with existing results providing sharp characterizations only for Gaussians or in low dimension~\citep{Chewi+22Sampling1D, Chewi+23QueryLower}.

\section{Conclusion}

In this work, we have shown that high-accuracy guarantees---$\polylog(1/\delta)$ rates---are achievable for sampling, provided that the stochastic gradients have light tails.
Moreover, via an information-theoretic argument, we have shown that light tails are necessary for such a result.
In fact, as a by-product of our analysis, we identified that the optimal dependence is $\Theta(1/\delta)$ if the stochastic gradients are only assumed to have a bounded variance.
\scedit{We then improved the state-of-the-art for high-accuracy sampling from finite-sum potentials.}

Several open questions remain, of which we list two: (1) Can the lower bound be extended to capture dependence on other problem parameters, such as the dimension $d$? (2) \scedit{What is the optimal complexity} in the finite-sum setting?

\arxiv{
  \section*{Acknowledgments}
  We thank Sam Power for bringing to our attention useful references.
  We acknowledge support from AFOSR through award FA9550-25-1-0375, Simons Foundation and the NSF through awards DMS-2031883 and PHY-2019786, and DARPA AIQ award.
 CD is supported by a Simons Investigator Award, a Simons Collaboration on Algorithmic Fairness, ONR MURI grant N00014-25-1-2116, and ONR grant N00014-25-1-2296.
}

\bibliography{ref.bib}

\newpage

\appendix

\tableofcontents

\section{Implementation of the approximate proximal oracle}\label{app:prox}
In this section, we show how to implement the approximate proximal oracle (\cref{asmp:prox_apx}) using a stochastic gradient oracle.

\begin{proof}[\pfref{lem:implement-prox}]
Our goal is to approximately compute
\begin{align*}
    \xhat\approx \arg\min_{x\in\R^d} \prn[\big]{ f(x)+\frac{1}{2\eta}\nrm{x-\xz}^2 }\,.
\end{align*}
Let $\xstar$ be an optimal solution to the minimization problem. We denote $m_s\deq \beta_s^{1/(1+s)}$ and assume that $\eta\leq \frac{1}{2m_s}$.

We consider the following linearized update rule: Let $X_0=\xz$, and
\begin{align*}
    X_{k+1}=\frac{X_k-\eta g_k+\xz}{2}\,,\quad 
    g_k\sim \Ogradn{X_k}\,,\qquad \forall k\geq 0\,.
\end{align*}
Note that $\xstar+\eta \nabla f(\xstar)=\xz$. Therefore,
\begin{align*}
    2\nrm{X_{k+1}-\xstar}=&~\nrm{X_k-\eta g_k-\xstar+\eta \nabla f(\xstar)} \\
    \leq&~ \nrm{X_k-\xstar-\eta(\nabla f(X_k)-\nabla f(\xstar))}+\eta\nrm{g_k-\nabla f(X_k)}\,.
\end{align*}
In the following we denote
\begin{align*}
    \Delta_k\deq \nrm{X_k-\xstar}\,, \qquad
    E_k\deq\nrm{g_k-\nabla f(X_k)}\,.
\end{align*}
Note that by
\cref{ass:holder},
\begin{align*}
    \nrm{X_k-\xstar-\eta(\nabla f(X_k)-\nabla f(\xstar))} 
    \le &~\nrm{X_k-\xstar}+\eta\nrm{\nabla f(X_k)-\nabla f(\xstar)} \\
    \leq&~ \Delta_k+\eta\beta_s\Delta_k^{s}\,.
\end{align*}
When $s\in(0,1)$, we can use AM--GM inequality to derive $\eta \beta_s\Delta^{s}\leq \frac{1}{2}\Delta+\eta m_s$. It is also straightforward to verify this inequality holds for $s\in\crl{0,1}$.

Then, it holds that
\begin{align*}
    \Delta_{k+1}\leq \frac{\frac32\Delta_k+\eta (m_s+E_k)}{2}\,, \qquad \forall k\geq 0\,.
\end{align*}
Applying this inequality recursively, we get
\begin{align*}
    \Delta_k\leq \prn[\Big]{\frac{3}{4}}^k \Delta_0+2\eta m_s+\frac{\eta}{2}\sum_{i=1}^{k} \prn[\Big]{\frac{3}{4}}^{i-1}E_{k-i}\,.
\end{align*}
Note that $\PP(Y\geq 2y)\leq \frac{1}{y}\En(Y-y)_+$ for $y>0$, and hence
\begin{align*}
    \PP\prn[\Big]{ \Delta_k\geq 2\prn[\Big]{\frac{3}{4}}^k \Delta_0 +4\eta(m_s+M) }\leq&~ \frac{1}{8M}\sum_{i=1}^{k} \prn[\Big]{\frac{3}{4}}^{i-1}\En(E_{k-i}-M)_+ \\
    \leq&~ \frac{1}{2}\eps_n(M)\,.
\end{align*}
Note that
\begin{align*}
    \Delta_0=&~\nrm{x_0-\xstar}
    =\eta\nrm{\nabla f(\xstar)} 
    \leq \eta \nrm{\nabla f(x_0)}
    +\eta \nrm{\nabla f(x_0)-\nabla f(\xstar)} \\
    \leq&~ \eta G+\eta\beta_s\nrm{x_0-\xstar}^s 
    \leq \eta G+\eta m_s+\frac12\Delta_0,
\end{align*}
and hence $\Delta_0\leq 2\eta G+2\eta m_s$.
In particular, when $k\geq 10\log(4G/(M+m_s))$, we know
\begin{align*}
    \PP\prn*{ \Delta_k\geq 5\eta(m_s+M)}\leq \eps_n(M)\,.
\end{align*}
Finally, note that
\begin{align*}
    \nrm{X_k+\eta \nabla f(X_k)-x_0}
    \leq&~ \nrm{X_k-\xstar}+\eta \nrm{\nabla f(X_k)-\nabla f(\xstar)} \\
    \leq&~ \Delta_k+\eta \beta_s\Delta_k^{s}
    \leq \frac32\Delta_k+\eta m_s,
\end{align*}
and hence we know
\begin{align*}
    \PP\prn*{ \nrm{X_k+\eta \nabla f(X_k)-x_0}\geq 10\eta(m_s+M)}\leq \eps_n(M)\,.
\end{align*}
\end{proof}

\section{Technical tools}
\begin{lemma}\label{lem:Gaussian-trunc}
Let $\lambda>0$, $B>0$. Then
\begin{align*}
    \En_{Z\sim \normal{0,\sigma^2}}\brk{e^{\lambda(\abs{Z}-B)_+}-1}\leq 2e^{\frac12\lambda^2\sigma^2}\,.
\end{align*}
Further, when $B\geq 2\max\crl{\lambda \sigma^2,\sigma}$, we can bound
\begin{align*}
    \En_{Z\sim \normal{0,\sigma^2}}\brk{e^{\lambda(\abs{Z}-B)_+}-1}\leq e^{-\frac{B^2}{8\sigma^2}}.
\end{align*}
\end{lemma}

\begin{proof}
By rescaling, we may assume $\sigma=1$. Then, we can upper bound
\begin{align*}
    \En_{Z\sim \normal{0,\sigma^2}}\brk{e^{\lambda(\abs{Z}-B)_+}-1}
    \leq&~ \sqrt{\frac{2}{\pi}}\int_{B}^{\infty} e^{\lambda (z-B)-\frac12z^2}\,dz \\
    =&~ \sqrt{\frac{2}{\pi}}\int_{B}^{\infty} e^{\frac12\lambda^2-\lambda B-\frac12(z-\lambda)^2}\,dz
    =2e^{\frac12\lambda^2-\lambda B}\Phi(B-\lambda)\,,
\end{align*}
where $\Phi(w)=\frac{1}{\sqrt{2\pi}}\int_{w}^\infty e^{-\frac12z^2}\,dz$. The first inequality then follows from $\Phi(B-\lambda)\leq 1$. Further, using the inequality $\Phi(w)\leq \frac{1}{\sqrt{2\pi}w}e^{-\frac12w^2}$ for $w>0$, we can bound
\begin{align*}
    \En_{Z\sim \normal{0,\sigma^2}}\brk{e^{\lambda(\abs{Z}-B)_+}-1}
\leq&~ 2e^{\frac12\lambda^2-\lambda B}\Phi(B-\lambda) \\
\leq&~ e^{\frac12\lambda^2-\lambda B-\frac{1}{8}B^2}\leq e^{-\frac{1}{8}B^2}\,.
\end{align*}
\end{proof}

The following lemma is standard~\citep{chen2026high}.
\begin{lemma}\label{lem:Gaussian-s}
Suppose that $\eta>0$ and $0\leq \lambda\leq \frac{d^{1-s}}{4s\eta^s}$. Then, it holds that
\begin{align*}
    \En_{W\sim \normal{0,\eta\Id}} \exp\prn*{\lambda\nrm{W}^{2s}}\leq \exp\prn*{2(\eta d)^s \lambda}\,.
\end{align*}
\end{lemma}

For two probability measures $\mu$, $\nu$, and $\ell > 1$, we write $\Dren[\ell]{\mu}{\nu} \deq \En_\mu (\frac{\mu}{\nu})^{\ell-1}-1$.

\begin{lemma}\label{lem:Renyi-to-diff}
For any $\ell>1$, it holds that
\begin{align}
    \max\crl{ \Dren[\ell]{\mu_f}{\mu_g}, \Dren[\ell]{\mu_g}{\mu_f} }
    \leq \En_{x\sim \mu_f}\brk{e^{2\ell\,\abs{f(x)-g(x)}}-1}\,.
\end{align}
Furthermore, it holds that
\begin{align}
    \Dtv{\mu_f}{\mu_g}
    \leq \En_{x\sim \mu_f}\brk{e^{2\,\abs{f(x)-g(x)}}-1}\,.
\end{align}
\end{lemma}
\begin{proof}
By definition, we can write
\begin{align*}
    Z_g
    = \int_{\RR^d} e^{-f(x)+(f(x)-g(x))}\,dx
    = Z_f \cdot \mathbb{E}_{x\sim \mu_f}e^{f(x)-g(x)},
\end{align*}
and hence
\begin{align*}
    \frac{\mu_f(x)}{\mu_g(x)}=\exp(g(x)-f(x))\frac{Z_g}{Z_f}=e^{g(x)-f(x)}\En_{\mu_f}[e^{f-g}]\,.
\end{align*}
Therefore, we have
\begin{align*}
    1+\Dren[\ell]{\mu_f}{\mu_g}=&~\En_{\mu_f}\prn[\big]{\frac{\mu_f}{\mu_g}}^{\ell-1}
    = \prn*{\En_{\mu_f}[e^{f-g}]}^{\ell-1}\cdot \En_{\mu_f}[e^{(\ell-1)(g-f)}] \\
    \leq&~ \prn[\big]{\En_{\mu_f}e^{(\ell-1)\abs{f-g}}}^2
    \leq \En_{\mu_f}[e^{2\ell\abs{f-g}}]\,.
\end{align*}
Similarly,
\begin{align*}
    1+\Dren[\ell]{\mu_g}{\mu_f}=&~\En_{\mu_f}\prn[\big]{\frac{\mu_g}{\mu_f}}^{\ell}= \prn*{\En_{\mu_f}[e^{f-g}]}^{-\ell}\cdot \En_{\mu_f}[e^{\ell(f-g)}] 
    \leq \En_{\mu_f}[e^{-\ell(f-g)}]\cdot \En_{\mu_f}[e^{\ell(f-g)}]\\
    \leq&~ \prn[\big]{\En_{\mu_f}e^{\ell\abs{f-g}}}^2
    \leq \En_{\mu_f}[e^{2\ell\abs{f-g}}]\,.
\end{align*}
Combining both inequalities completes the proof of the first inequality.

To prove the second inequality, we note that
\begin{align*}
    2\Dtv{\mu_f}{\mu_g}=&~\En_{\mu_f}\abs[\Big]{\frac{\mu_\scedit{g}}{\mu_\scedit{f}}-1}
    =\En_{x\sim \mu_f}\abs[\Big]{\frac{e^{f(x)-g(x)}}{\En_{\mu_f}[e^{f-g}]}-1} \\
    \leq&~ \frac{1}{\En_{\mu_f}[e^{f-g}]}\,\prn[\big]{ \En_{\mu_f}\abs{e^{f-g}-1}+\abs{\En_{\mu_f}[e^{f-g}]-1} } \\
    \leq&~ \frac{2}{\En_{\mu_f}[e^{f-g}]} \En_{\mu_f}\abs{e^{f-g}-1}\,.
\end{align*}
Note that $\abs{e^w-1}\leq e^{\abs{w}}-1$ and $\frac{1}{\En_{\mu_f}[e^{f-g}]}\leq \En_{\mu_f}[e^{g-f}]\leq \En_{\mu_f}[e^{\abs{f-g}}]$, so we can deduce
\begin{align*}
    \Dtv{\mu_f}{\mu_g}
    \leq \En_{\mu_f}[e^{\abs{f-g}}]\,\prn[\big]{ \En_{\mu_f}[e^{\abs{f-g}}]-1 }
    \leq \En_{\mu_f}[e^{2\abs{f-g}}-1]\,.
\end{align*}
\end{proof}

\begin{lemma}\label{lem:Renyi-to-cov}
For any $f:\cX\to [0,1]$, it holds that
\begin{align*}
    \En_{p}[f]-M\En_q[f]\leq \inf_{\lambda>1} M^{-(\lambda-1)}\,(1+\Dren{p}{q})\,.
\end{align*}
\end{lemma}

\begin{proof}
By definition,
\begin{align*}
    \En_{p}[f]-M\En_q[f]=\En_q\brk[\Big]{\prn[\Big]{\frac{dp}{dq}-M}\,f}
    \leq \En_q \prn[\Big]{\frac{dp}{dq}-M}_+\,.
\end{align*}
Note that for any random variable $Y\geq 0$, we have
\begin{align*}
    \En(Y-M)_+=\En[\indic{\{Y>M\}}\,(Y-M)_+]\leq \bbP(Y>M)^{1-\frac{1}{\lambda}}\prn*{\En[(Y-M)_+^\lambda]}^{\frac{1}{\lambda}}\leq \frac{\En[Y^\lambda]}{M^{\lambda-1}}\,.
\end{align*}
Combining these inequalities and taking infimum over $\lambda>1$ completes the proof.
\end{proof}

\begin{lemma}[Sub-additivity for TV distance]
\label{lem:TV-chain}
Suppose that $X_1\to \cdots\to X_T$ is a Markov chain. Given a family of transition kernels $\rho=(\rho_t: \cX\to \Delta(\cX))_{t\in [T]}$, we let $\PP_\rho$ be the law of $X_1,\ldots,X_T$ under $X_1\sim \rho_1$, $X_t\sim\rho_t(\cdot\mid X_{t-1})$. Then, for any families of transition kernels $\rho, \rho'$, it holds that
\begin{align*}
  \Dtv{\PP_\rho}{\PP_{\rho'}}\leq 
\sum_{t=1}^{T} \En_{X_{t-1}\sim \PP_{\rho}}\brk*{\Dtv{\rho_t(\cdot\mid X_{t-1})}{\rho'_t(\cdot\mid X_{t-1})}}\,,
\end{align*}
where we regard $X_0= {\perp}$ and $\rho_1(\cdot\mid \perp)=\rho_1$.
\end{lemma}

\begin{lemma}\label{lem:pclip-triangle}
Suppose that $Y$ is a random variable such that $\En[Y]=0$. Then for any $B>0$, $X\in\RR$,
\begin{align*}
    \abs{X-\En\pclip{X+Y}}\leq&~\trunc[B/2]{X} + \min\crl{2B,\En\trunc[B/2]{Y}}.
\end{align*}
\end{lemma}
\begin{proof}
First, note that $\abs{\pclip{X}-\pclip{X+Y}}\leq 2B$, we know
\begin{align*}
    \abs{X-\En\pclip{X+Y}}\leq&~ \abs{X-\pclip{X}}+\abs{\En[\pclip{X}-\pclip{X+Y}]} \\
    \leq&~ \trunc{X} + 2B.
\end{align*}
On the other hand, we know $X=\En[X+Y]$, and hence
\begin{align*}
    \abs{X-\En\pclip{X+Y}}\leq&~ \En\abs{X+Y-\pclip{X+Y}} \\
    =&~\En\trunc{X+Y}
    \leq \trunc[B/2]{X} + \En\trunc[B/2]{Y}.
\end{align*}
Combining both inequalities 
completes the proof.
\end{proof}

\section{Proofs from \cref{sec:upper}}

\subsection{\pfref{thm:gaussian_tilt_stoc}}

\newcommand{\Wbar}{\wb{W}}

Without loss of generality we only consider the case $n=1$.
We denote by $P_x$  the joint distribution of $(r,z,g)$ under \eqref{eq:def-stoc-W}.

By~\cref{thm:fors}, the output of~\cref{alg:fors} with the specified choices samples from $\widehat\nu$, such that
\begin{align*}
    \log \wh\nu(x)-\log q(x)=\const+\En_{(r,z,g)\sim P_x}\pclip{W_{r,z,g,x}}\,.
\end{align*}
We denote $\Wbar_{r,z,x}\deq \tri{\gamp{\lr}, u - \nabla f(\gam{\lr})}$, and we know
\begin{align*}
    W_{r,z,g,x}-\Wbar_{r,z,x}=\tri{\gamp{\lr}, \nabla f(\gam{\lr}) - g}\,
\end{align*}
has mean zero under $g\sim \Ograd{\gam{\lr}}$.
On the other hand, we know
\begin{align*}
    \En_{r,z} \Wbar_{r,z,x}=
    \En_{r,z}\tri{\gamp{\lr}, u - \nabla f(\gam{\lr})}
    =&~-f(x)+\tri{u,x}+\const \\
    =&~ \log \nu(x)-\log q(x)+\const\,.
\end{align*}
Then, using \cref{lem:pclip-triangle}, we have
\begin{align*}
    &~\abs{\log\rgo(x)-\log \wh\nu(x)-\const} \\
    \leq &~ \En_{r,z}\abs*{ \Wbar_{r,z,x}-\En_{g\sim \Ograd{\gam{\lr}}}\pclip{W_{r,z,g,x}}  } \\
    \leq&~ \En_{r,z} \trunc[B/2]{\Wbar_{r,z,x}}+\En_{r,z} \min\crl*{2B,\En_{g\sim \Ograd{\gam{\lr}}} \trunc[B/2]{W_{r,z,g,x}-\Wbar_{r,z,x}}} 
    =:V(x).
\end{align*}

Then, using \cref{lem:Renyi-to-diff}, 
\begin{align}\label{eq:TV-to-diff}
\begin{aligned}
    \Dtv{\nu}{\nuhat}
    \leq&~ \En_{x\sim \nuhat}\brk{ e^{2V(x)}-1 }
    \leq e^{2B}\En_{x\sim q}\brk{ e^{2V(x)}-1 }.
\end{aligned}
\end{align}
where the second inequality uses $\frac{d\nuhat}{dq}(x)\leq e^{2B}$ for $x\in\RR^d$. Using the definition of $V$, $2e^{u+v}\leq e^{2u}+e^{2v}$ and the convexity of $w\mapsto e^w$, we also know
\begin{align*}
    2e^{2V(x)}
    \leq&~ \En_{r,z} \exp\prn*{ 2\trunc[B/2]{\tri{\gamp{\lr}, u-\nabla f(\gam{\lr})}} } \\
    &+ \En_{r,z} \exp\prn*{ 2\min\crl*{2B,\En_{g\sim \Ograd{\gam{\lr}}} \trunc[B/2]{\tri{\gamp{\lr}, \nabla f(\gam{\lr})-g}}} }.
\end{align*}
Now, note that for any fixed $\lr\in[0,1]$ under $x\sim q=\normal{\xhat,\eta\Id}$ and $z\sim P=\normal{0,\eta\Id}$, $[\gam{\lr};\gamp{\lr}]$ are jointly distributed as
\begin{align}\label{eq:ind-Gaussian}
    [\gam{\lr};\gamp{\lr}]\sim \normal{\begin{bmatrix} \xhat \\ 0 \end{bmatrix},\; \begin{bmatrix} \eta\Id &  \\ & (\pi/2)^2\eta\Id \end{bmatrix}
    }\,.
\end{align}
This implies that 
\begin{align*}
    &~2(1+e^{-2B}\Dtv{\nu}{\nuhat})
    \leq 2\En_{x\sim q} e^{2V(x)} \\
    \leq&~ \En_r \En_{x\sim q,\,z\sim \Neta} \exp\prn*{ \trunc[B/2]{\tri{\gamp{\lr}, u-\nabla f(\gam{\lr})}} } \\
    &+ \En_r \En_{x\sim q,\,z\sim \Neta} \exp\prn*{ \min\crl*{2B,\En_{g\sim \Ograd{\gam{\lr}}} \trunc[B/2]{\tri{\gamp{\lr}, \nabla f(\gam{\lr})-g}}} } \\
    =&~ \En_{x\sim q,\, Z_1\sim \normal{0,(\pi/2)^2\eta\Id}}\exp\prn*{ \trunc[B/2]{\tri{Z_1, u-\nabla f(Z)}} } \\
    &~+ \En_{x\sim q,\, Z_1\sim \normal{0,(\pi/2)^2\eta\Id}} \exp\prn*{ \min\crl*{2B,\En_{g\sim \Ograd{x}} \trunc[B/2]{\tri{Z_1, \nabla f(x)-g}}} }.
\end{align*}

Therefore, by \cref{lem:clip-tail}, there is a constant $c_1$ such that as long as $\frac{1}{\eta}\geq c_1 M^2(\log(1/\delta)+\lambda)$, it holds that for any $x\in\RR^d$ such that $\eps(M;x)\leq C\deq \frac{e^{4B}-1}{8}$,
\begin{align*}
    \En_{Z_1\sim \normal{0,(\pi/2)^2\eta\Id}}\exp\prn[\big]{2\En_{g\sim \Ograd{Y}}\trunc[B/2]{\abs{\tri{Z_1,g-\nabla f(\scedit{x})}}}}-1\leq \delta+8\eps(M)\,.
\end{align*}
This immediately implies that 
\begin{align*}
    \MoveEqLeft\En_{x\sim q,\, Z_1\sim \normal{0,(\pi/2)^2\eta\Id}} \exp\prn*{ 2\min\crl*{2B,\En_{g\sim \Ograd{x}} \trunc[B/2]{\tri{Z_1, \nabla f(x)-g}}} }-1 \\
    \leq&~ \delta+ 8\En_{x\sim q}\min\crl*{\eps(M;x), C}.
\end{align*}
By \cref{lem:holder-tail} and \cref{cor:tilt-coverage}, there is a constant $c_2$ such that as long as
\begin{align*}
    \frac{1}{\eta^{1+s}}\geq c_2 \beta_s^2\,\prn[\Big]{d^s\log(1/\delta)+\frac{s}{d^{1-s}}\log^2(1/\delta)} + c_2 \,\prn[\big]{\epsprox^2 \log(1/\delta)}^{1+s}\,,
\end{align*}
it holds that
\begin{align*}
    \En_{x\sim q,\, Z_1\sim \normal{0,(\pi/2)^2\eta\Id}}\exp\prn*{ 2\trunc[B/2]{\tri{Z_1, u-\nabla f(Z)}} }-1\leq \delta,
\end{align*}
and $\En_q[f]\leq e\En_\nu[f]+\delta$ for any bounded function $f:\RR^d\to [0,1]$. This immediately implies that 
\begin{align*}
    \En_{x\sim q}\min\crl*{\eps(M;x), C}
    \leq e\En_{x\sim \nu}\min\crl*{\eps(M;x), C}+\delta.
\end{align*}
Combining the inequalities above and rescale $\delta\leftarrow \frac{\delta}{3}$ completes the proof.
\jmlrQED

\begin{lemma}\label{lem:clip-tail}
Suppose that $C>1$ is a constant, $\sigma M\leq \frac{B}{2}$ and $\lambda\leq \min\crl{ \frac{B}{2C^2\sigma^2M^2}, \frac{1}{2CM\sigma} }$. Then, as long as $\eps(M;x)\leq C$, it holds that
\begin{align*}
    \En_{Z\sim \Nsigma}\exp\prn*{\lambda \En_{g\sim \Ograd{x}}\trunc{\abs{\tri{Z,\nabla f(x)-g}}}}-1
    \leq e^{-\frac{B^2}{8\sigma^2M^2}}+8\eps(M;x)\,.
\end{align*}
\end{lemma}

\newcommand{\Qbar}{\wb{Q}}

\begin{proof}
We denote $P=\Nsigma$ and let $Q$ be the distribution of $v=g-\nabla f(x)$. 
Without loss of generality, we assume the support of $Q$ does not contain $0$, and define $\mmt=\En_{v\sim Q}\nrm{v}$. Note that we assume $\eps(M;x)\leq 1$, i.e.,
\begin{align*}
    (C-1)M\geq \En_{v\sim Q}\brk*{\nrm{v}\indic\crl*{\nrm{v}>M}}=\mmt- \En_{v\sim Q}\brk*{\nrm{v}\indic\crl*{\nrm{v}\leq M}}\geq \mmt-M,
\end{align*}
i.e., $\mmt\leq CM$.
Define $\alpha_v=\frac{\mmt+M}{\nrm{v}+M}$. 
Then we know $\En_{v\sim Q}[1/\alpha_v]=1$, and hence we can consider the distribution $\Qbar$ over $\RR^d$ such that $\frac{d\Qbar}{dQ}(g)=1/\alpha_v$. We can rewrite
\begin{align*}
    I\deq&~ \En_{Z\sim P}\exp\prn[\big]{\lambda \En_{v\sim Q}\trunc{\abs{\tri{Z, v}}}}-1 \\
    =&~\En_{Z\sim P}\exp\prn[\big]{\lambda \En_{v\sim \Qbar}[\alpha_v\trunc{\abs{\tri{Z, v}}}]}-1 \\
    \leq&~\En_{Z\sim P,v\sim \Qbar}\exp\prn[\big]{\lambda \alpha_v\trunc{\abs{\tri{Z, v}}}}-1\,.
\end{align*}
Then, using \cref{lem:Gaussian-trunc}, we can obtain the following upper bounds:

(1) When $B\geq 2\max\crl{\nrm{v}\sigma, \lambda \alpha_v^2\nrm{v}^2 \sigma^2}$, i.e., when $\nrm{v}\leq \frac{B}{2\sigma}$ \emph{and} $\mmt+M \leq \frac{\sqrt{B/(2\lambda)}}{\sigma}$, it holds that
\begin{align*}
    \En_{Z\sim P}\exp\prn*{\lambda \alpha_v\trunc{\abs{\tri{Z, v}}}}-1\leq e^{-\frac{B^2}{8\sigma^2\nrm{v}^2}}\,.
\end{align*}

(2) For any $g\neq 0$, it holds that
\begin{align*}
    \En_{Z\sim P}\exp\prn*{\lambda \alpha_v\trunc{\abs{\tri{Z, v}}}}-1\leq 2e^{\frac12\lambda^2\alpha_v^2\nrm{v}^2\sigma^2}\leq 2e^{\frac{1}{2}\lambda^2(\mmt+M)^2\sigma^2}\leq 4\,,
\end{align*}
where we use the condition $\lambda\leq \frac{1}{2CM\sigma}$ and the fact that $\mmt\leq CM$.

Now, we note that we assume $M\leq \frac{B}{2\sigma}$ \emph{and} $CM \leq \frac{\sqrt{B/(2\lambda)}}{\sigma}$. Then we can upper bound
\begin{align*}
    I\leq&~ \En_{v\sim \Qbar}\brk*{ \indic\crl{\nrm{v}\leq M}\prn*{ \En_{Z\sim P}\exp\prn*{\lambda \alpha_v\trunc{\abs{\tri{Z, v}}}}-1 }} \\
    &~+\En_{v\sim \Qbar}\brk*{ \indic\crl{\nrm{v}> M}\prn*{ \En_{Z\sim P}\exp\prn*{\lambda \alpha_v\trunc{\abs{\tri{Z, v}}}}-1 }} \\
    \leq&~ e^{-\frac{B^2}{8\sigma^2M^2}}+4\PP_{v\sim \Qbar}\prn*{\nrm{v}\geq M}\,.
\end{align*}
Finally, using $\frac{d\Qbar}{dQ}(g)=\frac{\nrm{v}+M}{\mmt+M}$, we have 
\begin{align*}
    \PP_{v\sim \Qbar}\prn*{\nrm{v}\geq M}
    \leq \frac{2}{M} \En_{v\sim Q}\brk*{\nrm{v}\indic\crl*{\nrm{v}\geq M}}\,.
\end{align*}
Combining these inequalities gives the desired upper bound.
\end{proof}

\begin{lemma}\label{lem:holder-tail}
Suppose that $Z_0\sim \Neta$, $Z_1\sim \Nsigma$, and $Y=\xhat+Z_0$. Suppose \cref{ass:holder} holds and $\nrm{\nabla f(\xhat)-u}\leq \epsprox$. Then it holds that 
\begin{align*}
    \En\exp\prn*{\lambda \trunc{\abs{\tri{Z_1,\nabla f(Y)-u}}}}-1
    \leq 2\exp\prn[\Big]{-\frac{B}{12}\min\crl[\Big]{ \sqrt{\frac{d^{1-s}}{s\eta^s\sigma^2\beta_s^2}},\, \frac{B}{\sigma^2\epsprox^2 + \sigma^2\beta_s^2(\eta d)^s} }}\,, \\
    \text{for}~0\leq \lambda\leq \frac16\min\crl[\Big]{ \sqrt{\frac{d^{1-s}}{s\eta^s\sigma^2\beta_s^2}},\, \frac{B}{\sigma^2\epsprox^2+\sigma^2\beta_s^2(\eta d)^s} }\,.
\end{align*}
\end{lemma}

\begin{proof}
We write 
\begin{align*}
    M_\lambda\ldef&~ \En\exp\prn*{\lambda \trunc{\abs{\tri{Z_1,\nabla f(Y)-u}}}}-1 \\
    \leq&~ \En\exp\prn*{\lambda \prn{\abs{\tri{Z_1,\nabla f(Y)-u}}-B}} \\
    \leq&~ 2e^{-\lambda B}\En\exp\prn[\Big]{\frac12\lambda^2\sigma^2\nrm{\nabla f(Y)-u}^2} \\
    \leq&~2e^{-\lambda B+\lambda^2\sigma^2\nrm{\nabla f(\xhat)-u}^2}\En\exp\prn*{\lambda^2\sigma^2\nrm{\nabla f(Y)-\scedit{\nabla}f(\xhat)}^2}\,.
\end{align*}
Using \cref{ass:holder} and \cref{lem:Gaussian-s}, it holds that as long as $\lambda^2\sigma^2\beta_s^2\leq \frac{d^{1-s}}{4s\eta^s}$, 
\begin{align*}
    \En\exp\prn*{\lambda^2\sigma^2\nrm{\nabla f(Y)-\nabla f(\xhat)}^2}
    \leq&~ \En\exp\prn*{\lambda^2\sigma^2\beta_s^2\nrm{Y-\xhat}^{2s}} 
    \leq \exp\prn*{2\lambda^2\sigma^2\beta_s^2(\eta d)^s}\,.
\end{align*}
Therefore, we have shown
\begin{align*}
    M_\lambda\leq 2\exp\prn*{ \lambda^2\sigma^2\epsprox^2 + 2\lambda^2\sigma^2\beta_s^2(\eta d)^s - \lambda B}\,, \qquad \text{for}~\lambda\leq \sqrt{\frac{d^{1-s}}{4s\eta^s\sigma^2\beta_s^2}}\,.
\end{align*}
Note that $\lambda \mapsto M_\lambda$ is an increasing function, and hence %
we can choose
\begin{align*}
    \lambda_\star=\min\crl[\Big]{ \sqrt{\frac{d^{1-s}}{4s\eta^s\sigma^2\beta_s^2}},\, \frac{B}{2\sigma^2\epsprox^2 +4\sigma^2\beta_s^2(\eta d)^s} }\,,
\end{align*}
so that for any $\lambda\leq \lambda_\star$, it holds that
\begin{align*}
    M_\lambda\leq M_{\lambda_\star}
    &\leq 2\exp\prn[\Big]{-\frac{B}{12}\min\crl[\Big]{ \sqrt{\frac{d^{1-s}}{s\eta^s\sigma^2\beta_s^2}},\, \frac{B}{\sigma^2\epsprox^2+\sigma^2\beta_s^2(\eta d)^s} }}\,.
\end{align*}
\end{proof}

\begin{corollary}\label{cor:tilt-coverage}
There is an absolute constant $c>0$ such that the following holds. Suppose $\nrm{u-\nabla f(\xhat)}\leq \epsprox$.

For any $\delta\in(0,\frac12]$, as long as $\frac{1}{\eta}\geq c\epsprox^2\log(1/\delta)$ and $\frac{1}{\eta^{1+s}}\geq c\beta_s^2 d^s\log(1/\delta)$, for any $f:\RR^d\to [0,1]$ it holds that
\begin{align*}
    \En_q[f]\leq e\En_\nu[f]+\delta.
\end{align*}
\end{corollary}

\begin{proof}
Recall that (in the proof of \cref{thm:gaussian_tilt_stoc})
we denote $\Wbar_{r,z,x}\deq \tri{\gamp{\lr}, u - \nabla f(\gam{\lr})}$ and
\begin{align*}
    \log \nu(x)-\log q(x)+\const=\En_{r,z} \Wbar_{r,z,x}.
\end{align*}
By \cref{lem:Renyi-to-diff}, it holds that for any $\ell\geq 1$, 
\begin{align*}
    \max\crl*{\Dren[\ell]{\nu}{q}, \Dren[\ell]{q}{\nu}}+1\leq&~ \En_{x\sim q}\En_{r,z} e^{2\ell\abs{\Wbar_{r,z,x}}} \\
    =&~\En_r \En_{x\sim q,\,z\sim \Neta} \exp\prn*{ 2\ell\abs{\tri{\gamp{\lr}, u-\nabla f(\gam{\lr})}} } \\
    =&~ \En_{x'\sim q,\,z'\sim \normal{0,(\pi/2)^2\eta\Id}} \exp\prn*{ 2\ell\abs*{\tri{z', u-\nabla f(x')}} },
\end{align*}
where the last line uses \cref{eq:ind-Gaussian}. Then, from our proof of \cref{lem:holder-tail}, we know that as long as $50\ell^2\eta\beta_s^2\leq \frac{d^{1-s}}{4s\eta^s}$, \sccomment{This should lead to one more condition on $\eta$ in the statement}
\begin{align*}
    \max\crl*{\Dren[\ell]{\nu}{q}, \Dren[\ell]{q}{\nu}}+1\leq&~ 
    \En_{x'\sim q,\,z'\sim \normal{0,(\pi/2)^2\eta\Id}} \exp\prn*{ 2\ell\abs*{\tri{z', u-\nabla f(x')}} } \\
    \leq&~ 2\exp\prn*{100\ell^2\eta\nrm{\nabla f(\xhat)-u}^2+100\ell^2\eta\beta_s^2(\eta d)^s}\,.
\end{align*}

Finally, by \cref{lem:Renyi-to-cov}, it holds that for any $f:\RR^d\to [0,1]$,
\begin{align*}
    \En_q[f]-e\En_{\nu}[f]\leq \inf_{\ell\geq 1} e^{1-\ell}(1+\Dren[\ell]{q}{\nu}).
\end{align*}
Then, we set $\ell_{\star}=\frac{1}{200\eta(\epsprox^2+\beta_s^2(\eta d)^s)}$. As long as $\ell_\star\geq 1$, we have $\En_q[f]-e\En_{\nu}[f]\leq e^{1-\frac{1}{2}\ell_\star}$. This is the desired result.
\end{proof}

\begin{corollary}\label{cor:chi-square}
There is an absolute constant $c>0$ such that the following holds. Suppose $\nrm{u-\nabla f(\xhat)}\leq \epsprox$.
As long as $\frac{1}{\eta}\geq c\epsprox^2$ and $\frac{1}{\eta^{1+s}}\geq c\beta_s^2 d^s$, it holds that $1+\Dchis{\nu}{q}\leq e^{c(\eta\epsprox^2+\eta^{1+s}\beta_s^2d^s)}$. 
\end{corollary}

\begin{proof}
It is straightforward to verify that for $V(x)\deq \En_{r,z} \Wbar_{r,z,x}$, it holds that
\begin{align*}
    1+\Dchis{\nu}{q}=\frac{\En_{x\sim q} e^{2V(x)}}{(\En_{x\sim q} e^{V(x)})^2}.
\end{align*}
Note that $\En_{x\sim q}[V(x)]=0$ by \cref{eq:ind-Gaussian}. Hence,
\begin{align*}
    1+\Dchis{\nu}{q}=&~\En_{x\sim q} e^{2V(x)}
    \leq \En_{x\sim q}\En_{r,z} e^{2\Wbar_{r,z,x}} \\
    =&~\En_r \En_{x\sim q,\,z\sim \Neta} \exp\prn*{ 2\tri{\gamp{\lr}, u-\nabla f(\gam{\lr})}}  \\
    =&~ \En_{x'\sim q,\,z'\sim \normal{0,(\pi/2)^2\eta\Id}} \exp\prn*{ 2 \tri{z', u-\nabla f(x')}} \\
    =&~ \En_{x'\sim q} \exp\prn[\Big]{ \frac{\pi^2}{2}\eta\, \nrm{ u-\nabla f(x')}^2}\,.
\end{align*}
The remaining proof is the same.
\end{proof}

\subsection{\pfref{thm:gaussian_tilt_stoc_0th}}

Without loss of generality we only consider the case $n=1$.
We denote by $P_x$ the joint distribution of $(z,v,v')$ under \eqref{eq:def-stoc-W-0th}. 

By~\cref{thm:fors}, the output of~\cref{alg:fors} with the specified choices samples from $\widehat\nu$, such that
\begin{align*}
   \log \wh\nu(x)-\log q(x)=\const+\En_{(z,v,v')\sim P_x}\pclip{W_{z,v,v',x}}\,.
\end{align*}
On the other hand, we know
\begin{align*}
   \En_{(z,v,v')\sim P_x}[W_{z,v,v',x}]
   =&~-f(x)+\tri{u,x}+\const \\
   =&~ \log \nu(x)-\log q(x)+\const\,.
\end{align*}
Then, using \cref{lem:pclip-triangle}, $\abs{\log\rgo(x)-\log \wh\nu(x)-\const}\leq V(x)$, where 
\begin{align*}
   V(x)
   \deq &~\En_{(z,v,v')\sim P_x}\trunc{W_{z,v,v',x}} \\
   \leq&~ \En_{z\sim q}\min\crl*{2B, \En_{v\sim \Oeval{x},\, v'\sim \Oeval{z}} \trunc[B/2]{v-f(x)+f(z)-v'}} \\
   &~+\En_{z\sim q} \trunc[B/2]{f(z)-f(x)-\tri{u,z-x}}\\
   \leq&~ \En_{z\sim q}\min\crl*{2B, \eps(B/4;x)+\eps(B/4;z)}
   +\En_{z\sim q} \trunc[B/2]{f(z)-f(x)-\tri{u,z-x}}\,.
\end{align*}
In the following, we denote $\Delta_{x,z}\ldef f(x)-f(z)-\tri{u,x-z}$.
Then, using \cref{lem:Renyi-to-diff}, 
\begin{align}
\begin{aligned}
   \Dtv{\nu}{\nuhat}\leq&~ \En_{x\sim \nuhat}\brk*{ \exp\prn*{2\En_{z\sim q}\trunc[B/2]{\Delta_{x,z}}+2\min\crl*{2B, \eps(B/4;x)+\eps(B/4;z)}}-1 } \\
   \leq&~ e^{2B}\En_{x,z\sim q}\brk*{ \exp\prn*{2\trunc[B/2]{\Delta_{x,z}}+2\min\crl*{2B, \eps(B/4;x)+\eps(B/4;z)}}-1 }
\end{aligned}
\end{align}
where we use $\frac{d\nuhat}{dq}(x)\leq e^{2B}$ for $x\in\RR^d$ and the convexity of $w\mapsto e^w$.
In the following, it remains to prove the following lemma (the rest of the proof then follows from the argument of \cref{thm:gaussian_tilt_stoc}). 
\jmlrQED
\begin{lemma}
Let $\Delta_{x,z}\ldef f(x)-f(z)-\tri{u,x-z}$ and $q=\normal{\xhat,\eta\Id}$. 
Suppose \cref{ass:holder} holds and $\nrm{\nabla f(\xhat)-u}\leq \epsprox$. Then it holds that 
\begin{align*}
   \En_{x,z\sim q}e^{\lambda\trunc{\Delta_{x,z}}}-1
   \leq 2\exp\prn[\Big]{-\frac{B}{32}\min\crl[\Big]{ \sqrt{\frac{d^{1-s}}{s\eta^s\eta\beta_s^2}},\, \frac{B}{\eta\epsprox^2 + \eta\beta_s^2(\eta d)^s} }}\,, \\
    \text{for}~0\leq \lambda\leq \frac1{16}\min\crl[\Big]{ \sqrt{\frac{d^{1-s}}{s\eta^s\eta\beta_s^2}},\, \frac{B}{\eta\epsprox^2+\eta\beta_s^2(\eta d)^s} }\,.
\end{align*}
\end{lemma}

\begin{proof}
We can express 
\begin{align*}
   \Delta_{x,z}=f(x)-f(z)-\tri{u,x-z}=\int_{0}^1 \tri{\gamp{\lr}, \nabla f(\gam{\lr})-u}\,d\lr,
\end{align*}
where $\gam{\lr}$ and $\gamp{\lr}$ are defined in \cref{eq:gamma}. Then, using the convexity of $w\to \trunc{w}$, we can upper bound
\begin{align*}
   \En_{x,z\sim q}e^{\lambda\trunc{\Delta_{x,z}}}
   \leq \En_{r\sim \unif([0,1])}\En_{x,z\sim q}e^{\lambda\trunc{\tri{\gamp{\lr}, \nabla f(\gam{\lr})-u}}}.
\end{align*}
Now, following the proof of \cref{thm:gaussian_tilt_stoc}, we know that for any fixed $r\in[0,1]$, under $x,z\sim q$, the vector $\gam{\lr},\gamp{\lr}$ are jointly distributed as
\begin{align*}
   [\gam{\lr};\gamp{\lr}]\sim \normal{\begin{bmatrix} \xhat \\ 0 \end{bmatrix}, \begin{bmatrix} \eta\Id &  \\ & (\pi/2)^2\eta\Id \end{bmatrix}
   }\,.
\end{align*}
Therefore, we have shown that
\begin{align*}
   \En_{x,z\sim q}e^{\lambda\trunc{\Delta_{x,z}}}
   \leq \En_{Z_0\sim \Neta, Z_1\sim \normal{0,(\pi/2)^2\eta\Id}}\exp\prn*{\lambda \trunc{\abs{\tri{Z_1,\nabla f(\xhat+Z_0)-u}}}}.
\end{align*}
Applying \cref{lem:holder-tail} completes the proof.
\end{proof}

\subsection{\pfref{thm:prox_1st}}

Let $\rho$ be the transition kernel on $\RR^d$ induced by the proximal sampler, i.e., $X'\sim \rho(\cdot\mid{}X)$ is generated by $Y'\sim \normal{X, \eta\Id}$ and $X'\sim \bar\pi^{X|Y=Y'}$. Then, $\rho$ induces a Markov chain $X_0,X_1,\cdots$ by $X_0\sim \mu_0, X_{n+1}\sim \rho(\cdot\mid{}X_n)$ for $n\geq 0$. For $n\geq 0$, let $\mu_n$ be the law of $X_n$.

Similarly, we let $\wh\rho$ be the transition with $\bar\pi^{X|Y=Y'}$ implemented via \cref{thm:gaussian_tilt_stoc}, and the induced Markov chain $X_0,X_1,\cdots$ is given by $X_0\sim \mu_0, X_{n+1}\sim \wh\rho(\cdot\mid{}X_n)$ for $n\geq 0$. For $n\geq 0$, let $\wh\mu_n$ be the law of $X_n$. 

Then, by \cref{lem:TV-chain} and data-processing inequality, it holds that
\begin{align}
    \Dtv{\wh\mu_N}{\mu_N}\leq \Dtv{\PP_{\wh\rho}}{\PP_{\rho}}\leq \sum_{n=0}^{N-1}\En_{X_n\sim \mu_n}\Dtv{\wh\rho(\cdot\mid{}X_n)}{\rho(\cdot\mid{}X_n)}.
\end{align}
In all cases, we use the following error analysis. By \cref{lem:nrm_grad} and the fact that $\Dchis{\mu_n}{\mu}\leq \Dchis{\mu_0}{\mu}$, it holds that with $G=O(\beta_s^{1/(1+s)}(d+\Delta+\log(N/\delta)))$,
\begin{align*}
    \max_{n\in [N]}\PP_{X\sim \mu_n}(\nrm{\nabla f(X_n)} \ge G_\delta) \le \frac{\delta}{10N}\,.
\end{align*}
Note that as long as $\nrm{\nabla f(X_n)}\leq G$, we can implement the proximal oracle at $X_n$ with $\epsprox=10(\beta_s^{1/(1+s)}+M))$ and success probability at least $1-\eps_n(M)$ by \cref{lem:implement-prox}, using
\begin{align*}
    O(n\log(G/\beta_s^{1/(1+s)}))=O(n\log A)\quad\text{queries to}~\Ograd{\cdot}\,.
\end{align*}

Therefore, by the choice of $\eta$ \eqref{eq:eta-1st}, we can implement the RGO via \cref{thm:gaussian_tilt_stoc} so that
\begin{align*}
    \Dtv{\wh\rho(\cdot\mid{}X_n)}{\rho(\cdot\mid{}X_n)}\leq \frac{\delta}{10N}+5\eps_n(M)
\end{align*}
as long as $\nrm{\nabla f(X_n)}\leq G$, using $O(n\log A)$ queries.
Then, by a conditioning argument, we see that
\begin{align*}
    \En_{X_n\sim \mu_n}\Dtv{\wh\rho(\cdot\mid{}X_n)}{\rho(\cdot\mid{}X_n)}
    &\leq \PP_{\mu_n}(\nrm{\scedit{\nabla}f(X_n)}\geq G)+\frac{\delta}{10N}+5\eps_n(M) \\
    &\leq \frac{\delta}{5N}+5\eps_n(M)\,.
\end{align*}
Therefore, taking summation over $k=0,1,\cdots,N-1$ gives 
\begin{align}
    \Dtv{\wh\mu_N}{\mu_N}\leq \frac{\delta}{5}+5N\eps_n(M).
\end{align}
Finally, we can set $n=\phi_M(\frac{\delta}{10N})$ so that $\eps_n(M)\leq \frac{\scedit{\delta}}{10N}$. Note that this implies
\begin{align*}
    \Dtv{\wh\mu_N}{\mu}\leq \frac{3}{4}\,\delta+\Dtv{\mu_N}{\mu}\,.
\end{align*}
Now, we apply results from~\citet{Chen+22ProxSampler} to bound $N$ such that $\Dtv{\mu_N}{\mu}\leq \frac{\delta}{4}$:
\begin{itemize}
    \item \textbf{LSI case.} Here, $N \asymp \frac{1}{\alpha \eta} \log \frac{\Dkl{\mu_0}{\mu}}{\delta^2}$.
    \item \textbf{PI case.} Here, $N \asymp \frac{1}{\alpha \eta} \log \frac{\Dchis{\mu_0}{\mu}}{\delta^2}$.
    \item \textbf{LC case.} Here, $N \asymp \frac{W_2^2(\mu_0,\mu)}{\eta\delta^2}$.
\end{itemize}
Plugging in the choice of $\eta$ in \eqref{eq:eta-1st} gives the desired results.
\jmlrQED

\begin{lemma}\label{lem:nrm_grad}
Suppose that \cref{ass:holder} holds and $\nu$ is a distribution such that
\begin{align*}
    \log(1+\Dchis{\nu}{\mu_f})\leq \Delta\,.
\end{align*}
Then it holds that for $\delta\in(0,1)$,
\begin{align*}
    \PP_{X\sim \nu}\prn[\Big]{ \nrm{\nabla f(X)}^2\geq \frac{64\beta_s^{2/(1+s)}}{d^{(1-s)/(1+s)}}\,\prn*{\Delta+d+\log(1/\delta)}  }\leq \delta.
\end{align*}
\end{lemma}

\begin{proof}
We follow the proof of ~\citet[Lemma 6.2.7]{Chewi26Book}. For any vector $v\in\RR^d$, we bound
\begin{align*}
    f(x+v)-f(x)-\tri{v,\nabla f(x)}
    =\int_{0}^1 \tri{v, \nabla f(x+rv)-\nabla f(x)}\,dr\leq \beta_s\,\nrm{v}^{1+s}\,, \qquad\forall x\in\RR^d\,.
\end{align*}
Then, we can bound
\begin{align*}
    \int_{\RR^d} e^{-f(x+v)}\,dx\geq \int_{\RR^d} e^{-f(x)-\tri{v,\nabla f(x)}-\beta_s\nrm{v}^{1+s}}\,dx\,.
\end{align*}
Re-organizing gives
\begin{align*}
    \En_{X\sim \mu_f}[e^{\tri{v,\nabla f(X)}}]\leq e^{\beta_s\nrm{v}^{1+s}}\,, \qquad \forall v\in\RR^d\,.
\end{align*}
For any $m\geq 0$ such that $\beta_s\leq \frac{d^{(1-s)/2}}{2(1+s)m^{(1+s)/2}}$, we can take expectation over $v\sim \normal{0,m\Id}$, and then \cref{lem:Gaussian-s} gives
\begin{align*}
    \En_{X\sim \mu_f}\exp\prn[\Big]{\frac{m}{2}\nrm{\nabla f(X)}^2}\leq \exp\prn[\big]{2\beta_s(md)^{(1+s)/2}}\,.
\end{align*}
Therefore, we choose $m>0$ such that $m^{1+s}=\frac{d^{1-s}}{16\beta_s^2}$, and then
\begin{align*}
    \En_{X\sim \nu}\exp\prn*{\frac{m}{4}\nrm{\nabla f(X)}^2}
    \leq&~ \sqrt{(1+\Dchis{\nu}{\mu_f})\En_{X\sim \mu_f}\exp\prn*{\frac{m}{2}\nrm{\nabla f(X)}^2}} \\
    \leq&~ \exp\prn[\Big]{\frac{1}{2}\Delta+\beta_s(md)^{(1+s)/2}}
    \leq \exp\prn[\Big]{\frac12\prn{\Delta+d}}\,.
\end{align*}
Applying Markov's inequality gives
\begin{align*}
    \PP_{X\sim \nu}(\nrm{\nabla f(X)}\geq G)
    \leq e^{-mG^2/4}\En_{X\sim \nu}\exp\prn*{\frac{m}{4}\nrm{\nabla f(X)}^2}
    \leq \delta\,,
\end{align*}
as long as $G^2\geq \frac{4}{m}\prn*{\Delta+d+\log(1/\delta)}$. This is the desired result.
\end{proof}

\subsection{\pfref{thm:prox_0th}}

The proof is very similar to the proof of \cref{thm:prox_1st}. 

By \cref{lem:nrm_grad} and the fact that $\Dchis{\mu_n}{\mu}\leq \Dchis{\mu_0}{\mu}$, it holds that with $G>0$ chosen as
\begin{align*}
    G^2=\frac{64\beta_s^{2/(1+s)}}{d^{(1-s)/(1+s)}}\,\prn*{\Delta+d+\log(10N/\delta)}\,,
\end{align*}
it holds that
\begin{align*}
    \max_{n\in [N]}\PP_{X\sim \mu_n}(\nrm{\nabla f(X_n)} \ge G_\delta) \le \frac{\delta}{10N}\,.
\end{align*}
Note that as long as $\nrm{\nabla f(X_n)}\leq G$, we can implement the proximal oracle at $X_n$ with $\epsprox=G$ by trivially returning $x=X_n$. 
Therefore, by the choice of $\eta$ \eqref{eq:eta-0th}, we can implement the RGO via \cref{thm:gaussian_tilt_stoc_0th} so that $\Dtv{\wh\rho(\cdot\mid{}X_n)}{\rho(\cdot\mid{}X_n)}\leq \frac{\delta}{10N}+10\eps_n(M)$ as long as $\nrm{\nabla f(X_n)}\leq G$, using $O(\scedit{n})$ queries. The rest of the proof is concluded as before.
\jmlrQED

\subsection{\pfref{thm:finite-sum}}

Denote
\begin{align*}
    \nu(x\mid y)\propto_x \exp\prn[\Big]{-f(x)-\frac{1}{2\eta}\,\nrm{x-y}^2}\,.
\end{align*}
We let $\PP_\star(\cdot)$ be the probability law of $(X_0,Y_0),\dotsc,(X_N,Y_N)$ induced by the proximal sampler, and $\En_\star[\cdot]$ be the corresponding expectation.

In the following, we choose $M>0$ as
\begin{align*}
    M^2=4\beta_1^2CK\eta\,(d+\log(K/\delta))+4\beta_1^2CK^2\eta^2 \scedit{\beta_1}\,(d+\Delta+\log(K/\delta))\,,
\end{align*}
where the constant $C>0$ is from \cref{lem:proximal-move},
and $A=M+\epsprox$.

For each $i\geq 0$, we consider each time step in the epoch $\cK_i\deq [iK,(i+1)K)$. By definition of the oracle $O_{k+1}$, as long as $\nrm{x-X_{iK}}\leq \frac{M}{2\beta_1}$, it holds that $\nrm{g \scedit{-\nabla f(x)}}\leq M$ deterministically under the oracle $O_{k+1}(x)$. Therefore, by \cref{thm:gaussian_tilt_stoc}, as long as
\begin{align}\label{pfeq:finite-sum-eta}
    \frac{1}{\eta}\gg \beta_1\sqrt{d\log(1/\delta)}+(A^2+M^2+\scedit{\beta_1})\log(1/\delta)\,,
\end{align}
and $\nrm{Y_{\scedit{k}}-\Xhat_{\scedit{k+1}}-\eta \nabla f(\Xhat_{\scedit{k+1}})}\leq \eta A$, it holds that our algorithm generates a sample $X_{k+1}$ following a distribution $\nuhat_{k+1}(\cdot\mid Y_{k},X_{iK})$ satisfying
\begin{align*}
    \Dtv{\nu(\cdot\mid Y_{k})}{\nuhat_{k+1}(\cdot\mid Y_{k},X_{iK})}\leqsim \delta+\PP_{X_{k+1}\sim \nu(\cdot\mid Y_{k})}\prn[\Big]{ \nrm{X_k-X_{iK}}\geq \frac{M}{2\beta_1} }\,.
\end{align*}
Note that \cref{pfeq:finite-sum-eta} can indeed be ensured by \cref{eq:finite-sum-eta}.
Then, taking expectation over $(Y_{k},X_{iK})\sim \PP_\star$, we have 
\begin{align*}
    \En_\star \Dtv{\nu(\cdot\mid Y_{k})}{\nuhat_{k+1}(\cdot\mid Y_{k},X_{iK})} \leqsim&~ \delta+\PP_\star\prn[\Big]{ \nrm{X_k-X_{iK}}\geq \frac{M}{2\beta_1} } \\
    &~+\PP_\star\prn[\big]{ \nrm{Y_{k}-\Xhat_{k+1}-\eta \nabla f(\Xhat_{k+1})} \scedit{\ge} \eta A }\,.
\end{align*}
By definition, $\Xhat_{k+1}=\Xhat_{iK}+Y_k-Y_{iK}$, and hence
\begin{align*}
    \nrm{Y_{k}-\Xhat_{k+1}-\eta \nabla f(\Xhat_{k+1})}
    =&~ \nrm{Y_{iK}-\Xhat_{iK}-\eta \nabla f(\Xhat_{k+1})} \\
    \leq&~ \nrm{Y_{iK}-\Xhat_{iK}-\eta \nabla f(\Xhat_{iK})}
    +\beta_1 \eta\, \nrm{\Xhat_{iK}-\Xhat_{k+1}} \\
    =&~ \nrm{Y_{iK}-\Xhat_{iK}-\eta \nabla f(\Xhat_{iK})}
    +\beta_1 \eta\, \nrm{Y_k-Y_{iK}}\,.
\end{align*}
Therefore, we can bound
\begin{align*}
    &~\PP_\star\prn[\big]{ \nrm{Y_{k}-\Xhat_{k+1}-\eta \nabla f(\Xhat_{k+1})}\geq \eta A } \\
    \leq&~ \PP_\star\prn[\big]{ \nrm{Y_{iK}-\Xhat_{iK}-\eta \nabla f(\Xhat_{iK})}\geq \eta \epsprox }+\PP_\star\prn[\Big]{\nrm{Y_k-Y_{iK}}\geq \frac{M}{\beta_1}}\,.
\end{align*}
By our definition of the proximal oracle $\Oprox{\cdot}$, the first term of the RHS is bounded by $\delta$.
Combining the inequalities and taking summation over $k$ and apply \cref{lem:TV-chain}, we know
\begin{align*}
    \Dtv{\mu_N}{\muhat_N}
    \leq&~ \sum_{i=0}^{\floor{N/K}} \sum_{k\in\cK_i\cap [N]} \En_\star \Dtv{\nu(\cdot\mid Y_{k})}{\nuhat_{k+1}(\cdot\mid Y_{k},X_{iK})} \\
    \leqsim&~ N\delta+ \sum_{i=0}^{\floor{N/K}} \sum_{k\in\cK_i\cap [N]}\brk[\Big]{ \PP_\star\prn[\Big]{ \nrm{X_k-X_{iK}}\geq \frac{M}{2\beta_1} } + \PP_\star\prn[\Big]{\nrm{Y_k-Y_{iK}}\geq \frac{M}{\beta_1}}} \\
    \leqsim&~ N\delta\,,
\end{align*}
where the last inequality follows from \cref{lem:proximal-move}.
\jmlrQED

\begin{lemma}\label{lem:proximal-move}
Suppose that \cref{ass:holder} holds and $\nu$ is a distribution such that
\begin{align*}
    \log(1+\Dchis{\nu}{\mu_f})\leq \Delta\,.
\end{align*}
Consider the Markov chain $X_0\to Y_0\to \cdots\to X_K\to Y_K$ generated by the proximal sampler. Then as long as $\eta\leq \frac{1}{C\beta_1\sqrt{dK}}$, it holds that for $\delta\in(0,1)$,
\begin{align*}
    \PP\prn[\Big]{ \max_{k\in[K]}\nrm{Y_k-Y_0}\geq R  }\leq \delta\,, \qquad
    \PP\prn[\Big]{ \max_{k\in[K]}\nrm{X_k-Y_0}\geq R  }\leq \delta\,,
\end{align*}
where $R>0$ is defined as ($C>0$ is an absolute constant):
\begin{align*}
    R^2\deq CK\eta(d+\log(K/\delta))+CK^2\eta^2\beta_1(d+\Delta+\log(K/\delta))\,.
\end{align*}
\end{lemma}

\begin{proof}
Denote $\nubar(\cdot\mid y)=\normal{\prox_{\eta f}(y), \eta\Id}$.
We consider the following distributions of Markov chain $X_0\to Y_0\to \cdots\to X_K\to Y_K$:

(1) $P$ is the distribution of the exact proximal sampler, i.e., $X_0\sim \nu$, and for each $k\in[K]$, $Y_k\mid X_{k}\sim \normal{X_k, \eta\Id}$ and $X_{k+1}\mid Y_k\sim \nu(\cdot\mid Y_k)$.

(2) $Q$ is the following distribution: $X_0\sim \nu$, and for each $k\in[K]$, $Y_k\mid X_{k}\sim \normal{X_k, \eta\Id}$ and $X_{k+1}\sim \nubar(\cdot\mid Y_k)$.

By \cref{cor:chi-square}, there is a constant $c_1>0$ such that as long as $\frac{1}{\eta}\geq c_1\beta_1\sqrt{d}$, it holds that
\begin{align*}
    1+\Dchis{\nu(\cdot\mid y)}{\nubar(\cdot\mid y)}\leq \exp\prn*{c\eta^2\beta_1^2d}\,, \qquad \forall y\in\RR^d\,.
\end{align*} 
Therefore, it is straightforward to verify that
\begin{align*}
    1+\Dchis{P}{Q}\leq (1+\max_{y\in\RR^d} \Dchis{\nu(\cdot\mid y)}{\nubar(\cdot\mid y)})^K\leq \exp\prn*{c\eta^2\beta_1^2dK}\leq O(1)\,.
\end{align*}
In the following, we denote $\Xbar_{k}=\prox_{\eta f}(Y_{k-1})$, $Z_k=X_{k}-\Xbar_{k-1}$, and $Z_k'=Y_k-X_k$. Then, we can express $Y_{k-1}=\Xbar_{k}+\eta\nabla f(\Xbar_{k})$, and hence $Y_k=Y_{k-1}+Z_k+Z_k'-\eta \nabla f(\Xbar_{k})$. Apply this recursively, we get
\begin{align*}
    Y_k-Y_0=\sum_{i=1}^k (Z_i+Z_i')-\eta \sum_{i=1}^k \nabla f(\Xbar_i)\,.
\end{align*}
Therefore, we can bound
\begin{align*}
    \nrm{Y_k-Y_0}\leq \nrm[\Big]{ \sum_{i=1}^k (Z_i+Z_i') }+\eta \sum_{i=1}^k \nrm{\nabla f(X_i)}+\eta \beta_1 \sum_{i=1}^k \nrm{Z_i}\,.
\end{align*}
Note that under $Q$, $\sum_{i=1}^k (Z_i+Z_i')\sim \normal{0, 2k\eta \Id}$ and hence for any $k\in[K]$, 
\begin{align*}
    Q\prn[\bigg]{ \nrm[\Big]{ \sum_{i=1}^k (Z_i+Z_i') }\geq \sqrt{2k\eta}\,(\sqrt{d}+2\sqrt{\log(1/\delta)}) }\leq \delta\,.
\end{align*}
In addition, $Q\prn[\big]{ \nrm{Z_i}\geq \sqrt{\eta}\,(\sqrt{d}+2\sqrt{\log(1/\delta)}) }\leq \delta$. Therefore, using the union bound, we get
\begin{align*}
    Q\prn[\bigg]{ \max_{k\in[K]}\nrm{Y_k-Y_0}\geq (\sqrt{2K\eta}+K\beta_1\eta\sqrt{\eta})\,(\sqrt{d}+2\sqrt{\log(2K/\delta)})+\eta \sum_{i=1}^K \nrm{\nabla f(X_i)} }\leq \delta\,.
\end{align*}
Note that $\eta\leq \frac{1}{\beta_1\sqrt{dK}}$.
Then, applying the change-of-measure argument, we get
\begin{align*}
    P\prn[\bigg]{ \max_{k\in[K]}\nrm{Y_k-Y_0}\geq C_1\sqrt{K\eta}\,(\sqrt{d}+\sqrt{\log(K/\delta)})+\eta \sum_{i=1}^K \nrm{\nabla f(X_i)} }\leq \delta\,.
\end{align*}
Finally, by \cref{lem:nrm_grad}, we can show $P(\nrm{\nabla f(X_i)} \scedit{\ge} C_2\sqrt{\beta_1(\Delta+d+\log(K/\delta))})\leq \frac{\delta}{K}$. Taking the union bound completes the proof of the first inequality. The second inequality follows similar by noting
\begin{align*}
    X_k-X_0=Z_0'+\sum_{i=1}^{k-1} (Z_i+Z_i')+Z_k-\eta \sum_{i=1}^k \nabla f(\Xbar_i)\,.
\end{align*}
\end{proof}

\section{Proofs from \cref{sec:lower}}

\subsection{\pfref{prop:Ograd-lower}}

Fix any $p\in[0,1]$ such that $p\psi(\sqrt\alpha\delta/p)+\psi(\sqrt\alpha\delta)\leq 1$. We denote $M=\frac{\sqrt\alpha\delta}{p}$.

We denote $\theta=\scedit{\delta/\sqrt\alpha}$.
Consider $\psi$-oracle $O_0$ and $O_\theta$: for any $x\in\RR$, $O_0(x)$ returns $\alpha x$ with probability 1, and $O_\theta(x)$ returns $\alpha x-M$ with probability $p$ and returns $\scedit{\alpha}x$ otherwise. Then, $O_0$ is a $\psi$-oracle for $f_0$, and $O_\theta$ is a $\psi$-oracle for $f_{\scedit{\theta}}$, because
\begin{align*}
    \En_{g\sim O_\theta(x)}\psi(\abs{g-f_{\scedit{\theta}}'(x)})=p\psi(M-\scedit{\alpha}\theta)+(1-p)\psi(\scedit{\alpha}\theta)\leq 1\,.
\end{align*}
Note that
\begin{align*}
    \Dtv{O_0(x)}{O_\theta(x)}=p\,,
\end{align*}
and hence by the sub-additivity of the TV distance (\cref{lem:TV-chain}), we have
\begin{align*}
    \Dtv{\Alg(O_0)}{\Alg(O_\theta)}\leq Tp\,.
\end{align*}
On the other hand, by our assumption, it holds that
\begin{align*}
    \Dtv{p_0}{\Alg(O_0)}\leq \frac{\delta}{10}\,, \qquad
    \Dtv{p_\theta}{\Alg(O_\theta)}\leq \frac{\delta}{10}\,.
\end{align*}
An elementary calculation also yields
\begin{align*}
    \Dtv{p_0}{p_\theta}=\Dtv{\normal{0,\scedit{\alpha^{-1}}}}{\normal{\theta,\scedit{\alpha^{-1}}}}=\Dtv{\normal{0,1}}{\normal{\delta,1}}\geq \frac{\delta}{3}\,.
\end{align*}
Therefore, by triangle inequality,
\begin{align*}
    \frac{\delta}{3}\leq \Dtv{p_0}{p_\theta}
    &\leq \Dtv{p_0}{\Alg(O_0)}+\Dtv{\Alg(O_0)}{\Alg(O_\theta)}+
    \Dtv{p_\theta}{\Alg(O_\theta)} \\
    &\leq \frac{\delta}{5}+Tp\,,
\end{align*}
and this implies $T\geq \frac{\delta}{10p}$. Taking \scedit{infimum} over $p\in(0,1]$ such that $p\psi(\sqrt\alpha\delta/p)+\psi(\sqrt\alpha\delta)\leq 1$ completes the proof.
\jmlrQED

\end{document}